\documentclass[12 pt]{article}
\usepackage{amssymb,amsmath,amsfonts,amsthm}
\usepackage{colortbl}
\newcommand\BBP{{\mathbb {P}}}

\newcommand\BBR{{\mathbb {R}}}

\newcommand\BBE{{\mathbb {E}}}

\newtheorem {Lemma}{Lemma}[section]

\newtheorem {Theorem}{Theorem}[section]
\newtheorem {Proposition}{Proposition}[section]
\newtheorem {Corollary}{Corollary}[section]

\theoremstyle{definition}
\newtheorem{Definition}{Definition}[section]

\newtheorem{Remark}{Remark}[section]

\newcommand\no{\noindent}

\newcommand\beq{\begin{equation}}
\newcommand\eeq{\end{equation}}

\def\no{\noindent}

\begin{document}
\title{Normal approximation for partial sums: general convex costs}
\author{J\'er\^ome Dedecker\footnote{J\'er\^ome Dedecker, Universit\'e Paris Cit\'e,  MAP5, UMR 8145 CNRS,
45 rue des  Saints-P\`eres,
F-75006 Paris, France.},
Florence Merlev\`ede\footnote{Florence Merlev\`ede, LAMA,  Univ Gustave Eiffel, Univ Paris Est Cr\'eteil, UMR 8050 CNRS,  \  F-77454 Marne-La-Vall\'ee, France.} 
and Emmanuel Rio \footnote{Emmanuel Rio, Universit\'e de Versailles, LMV, UMR 8100 CNRS,  \ 45 avenue des Etats-Unis, 
F-78035 Versailles, France.}
}
\date{}
\maketitle

\begin{abstract}
We provide non-asymptotic bounds and asymptotic limits for convex transport costs between the distribution of partial sums of independent and identically distributed  square integrable and centered random variables  and the normal distribution with mean zero and the same variance. The proof relies on controlling the transport cost by an appropriate ideal distance, combined with an adaptation of Lindeberg's method. The numerical constants and the asymptotic constants are explicit. \\

\end{abstract}

\section{Introduction}
\setcounter{equation}{0}

Let $ \psi$ be a non-negative,  even and convex function defined on ${\mathbb R}$. Define then the following transport cost between two probability laws $\mu$ and $\nu$ on ${\mathbb R}$: 
\[
\kappa_{\psi} ( \mu , \nu ) = \inf\{ \BBE ( \psi (X-Y) )\, : \, X \sim \mu,  Y \sim \nu \} 
\]
According to Theorem (8.1) in Major \cite{Ma78}, 
\beq
\label{Major78}
\kappa_{\psi} ( \mu , \nu ) = \int_0^1 \psi \big (  F^{-1} (u) - G^{-1} (u)  \big ) du \, .
\eeq
where $F$ and $G$ are the distribution functions of   $\mu$ and $\nu$ respectively, and $F^{-1}$ and $G^{-1}$ their generalized inverses. 
In this paper, we are interested in transport costs from the law of a sum of independent and  identically distributed (iid) centered real-valued random variables with finite variance to the Gaussian law with mean zero and the same variance. More precisely our aim is to find conditions ensuring that the transport cost
$\kappa_\psi$ between these two distributions remains bounded as $n$ tends to $\infty$. 
\par\smallskip
Let us now recall the known results in this area. In order to shorten the notations, we denote by $\kappa_p$ the transport cost in the case $\psi (x) = |x|^p$ (here $p\geq 1$) and $W_p = \kappa^{1/p}_p$ the associated Wasserstein distance. In the case $p=1$, \eqref{Major78} implies that 
\[
\kappa_1 (\mu , \nu) = \int_{\BBR} | F(x) - G (x) | dx \, .
\]
Consequently the results of Esseen \cite{E58} on the ${\mathbb L}^1$-norm of the difference between the distribution functions apply to $\kappa_1$ (see also  Ibragimov-Linnik \cite{IL71}, Section 5.3).
It follows from Esseen's results   that the costs $\kappa_1$ between the distribution of partial sums and the corresponding Gaussian distribution remain bounded as soon as the random variables have a finite absolute moment of order three.
More precisely, let  $(X_i)_{i \in {\mathbb Z}}$ be a sequence of iid real-valued centered random variables  in ${\mathbb L}^3$, with positive variance 
$\sigma^2$. Set $S_n = X_1 + X_2 + \cdots + X_n$ and let $G_{n\sigma^2}$ denote a Gaussian random variable with mean $0$ and variance 
$n\sigma^2$.  For a random variable $X$ denote by $P_X$ its distribution.  Esseen  \cite{E58} proved that
\beq
\label{Esseen58}
\sup_{n>0} \kappa_1 ( P_{S_n} , P_{G_{n\sigma^2} } ) \leq c_1 \sigma^{-2} \BBE ( |X_1|^3 )  
\eeq
for some positive universal constant $c_1$. According to Goldstein \cite{Gold2010}, $c_1 \leq 1$. 
Furthermore Esseen proved that the above costs converge to some  asymptotic constant (depending on the law of $X_1$) as $n$ tends
to $\infty$ (see Esseen \cite{E58}, Theorem 4.2).  
\par
Let us now discuss the case $p>1$. In this case, one cannot derive asymptotic results for the cost $\kappa_p$ in the central limit theorem from
results on the ${\mathbb L}^p$-norm of the difference between the distribution functions. B\'artfai \cite{Ba70} proved that, if the law of $X_1$ has a finite Laplace transform in a neighborhood of $0$, then, for any positive $\varepsilon$, 
\[
\lim_{n \rightarrow \infty} n^{-\varepsilon} \kappa_2 ( P_{S_n} , P_{G_{n\sigma^2} } ) = 0. 
\]
Later, Rio \cite{Rio09} extended \eqref{Esseen58} to the costs $\kappa_p$ for $p$ in $]1,2]$: he proved that, for any $p$ in $]1,2]$, 
there exists a positive constant $c_p$, depending only on $p$, such that 
\beq
\label{RioBobkov}
\sup_{n>0} \kappa_p ( P_{S_n} , P_{G_{n\sigma^2} } ) \leq c_p \sigma^{-2} \BBE ( |X_1|^{2+p}  )   .
\eeq
Next Bobkov \cite{Bo2018} proved that \eqref{RioBobkov} also holds true for $p>2$. On another hand, Rio \cite{Rio11} obtained asymptotic constants for the costs
$\kappa_p$ in the case $p$ in $]1,2]$ under the moment condition $\BBE ( |X_1|^{2+p}  ) < \infty$ (without conditions on the smoothness of 
the law of $X_1$). For $p>2$, Bobkov \cite{Bo2018} obtained asymptotic expansions for the costs $\kappa_p$ for random variables satisfying  
$\BBE ( |X_1|^{2+p}  ) < \infty$, under the Cr\'amer condition. 
\par\smallskip
In this paper, we are interested in giving upper bounds for  costs associated with general even and convex functions $\psi$ such that $x^{-2}\psi(x) \rightarrow 0$, as $x \rightarrow \infty$ (the precise definition of the class of costs is given in  Definition \ref{defclassPsi}).   In particular this class of functions includes the fonction $x \mapsto h(|x|)$ where $h(x) =   (x+1) \ln  (x+1) - x$. Our aim is to give sufficient conditions ensuring that the costs $\kappa_\psi ( P_{S_n} , P_{G_{n\sigma^2} } ) $ are uniformly bounded. As we shall see in Theorem \ref{coravecphi}, this will be the case as soon as $\BBE ( X_1^2 \psi ( |X_1| ) ) < \infty$.  Also, as a consequence of the introduced techniques, we will be  in position to get asymptotic constants 
for general convex costs (see Corollary \ref{asymptoticconstants}) and to provide tail inequalities for  $|S_n -  G_{n\sigma^2}| $ with  $G_{n\sigma^2}$ defined from $S_n$ via the quantile transformation,  as soon as the underlying random variables have a weak moment of order $2+p$ with $p \in ]1,2[$ (see Theorem \ref{UpperBoundHLtail}).  We shall also give upper bounds for the conditional value at risk associated with the partial sum $S_n$   (see Corollary \ref{corweakmoments} and Remark \ref{remarkcorcvar}).
We refer to Section  \ref{MR}, which is devoted to our main  results, for more details about the theoretical results.
\par\smallskip
We now give some insights on the proofs of our main results, which are given in Section \ref{sectionPR}, in the case $\sigma =1$. The main tools of the proofs are 
Proposition 2.1 and a suitable version of the Lindeberg method combined with a technique of acceleration of the convergence.
As in Rio \cite{Rio09}, if $\beta_3 := \BBE ( X_1^3) \not=0$, we replace step by step the initial random variables $X_k$ by iid random variables $Y_k$ with a finite Laplace transform such that $\BBE ( Y_1^3) = \beta_3$. More precisely, if $\Pi (\lambda)$ denotes a random variable  with Poisson distribution
of parameter $\lambda$, 
\beq \label{DefY1}
Y_1 = B_1 + N_1,
\eeq
where $B_1$ and $N_1$ are independent, $N_1$ is normally distributed with mean $0$ and variance $1/2$, and 
\beq \label{DefB1}
B_1 = 2 \beta_3 \Bigl( \Pi ( 1/(8\beta_3)^2 ) - 1/(8\beta_3^2) \Bigr) .
\eeq
\par\smallskip
Our methods allow us to get some numerical constants for this  first step.  In order to get general results with numerical constants, we then need to give estimates of the quadratic transportation cost between a Poisson distribution of parameter $\lambda$ and the normal distribution with mean $\lambda$ and variance $\lambda$ with a suitable numerical constant. Although it is known since a long time that the  transportation costs $\kappa_p$ 
for the normal approximation of the Poisson distribution remain bounded as $\lambda$ tends to $\infty$, up to our knowledge, numerical bounds do not exist, except in the case $p=1$.  In the case  $p=1$,  the general results of Peccati et al. (\cite{PSTU10}, Example 3.5) provide the upper bound
\beq
\label{Peccati10}
W_1 ( \mu_\lambda , \nu_\lambda) \leq 1 \text{ for any } \lambda >0,
\eeq
where $\mu_\lambda$ denotes the distribution of $\Pi (\lambda) - \lambda$ and $\nu_\lambda$  the normal distribution with mean $0$ and variance $\lambda$. In Proposition \ref{PropPoisson} of Section
\ref{sectionPR}, we give the more efficient upper bound 
\beq
\label{W2PoissonNormal}
W_2 ( \mu_\lambda , \nu_\lambda) \leq \sqrt{0.937} \leq 0.968 \text{ for any } \lambda>0,
\eeq
which implies that $W_1 ( \mu_\lambda , \nu_\lambda) \leq 0.968$, since $W_1 (\mu , \nu)  \leq W_2 (\mu , \nu)$. The proof of this result, given in Section \ref{sectionAppendix}, is based on a Tusn\'ady type Lemma proved in Massart \cite{Ma2002} together with a dyadic decomposition of a Poisson random variable. Notice that the asymptotic constants are much better: indeed, from the results of Esseen \cite{E58} in the case $p=1$ and Rio \cite{Rio11} for $p$ in $]1,2]$, 
\[
\label{WpAsPoisson}
\lim_{n\rightarrow \infty} W_p (\mu_n , \nu_n ) = \Vert (G^2-1)/6 + U  \Vert_p ,
\]
where $G$ is a standard normal and $U$ is a random variable with uniform distribution over $[-1/2,1/2]$, independent of $G$. 
In the case $p=2$, the above limit is equal to $\sqrt{5} \, /6 = 0.37267...$ (see Corollary 1.3 in Rio \cite{Rio11}). 
\par\smallskip
Concerning arbitrary laws, the numerical constants in \eqref{RioBobkov} have not been studied, except in the case $p=1$. In this paper, as a consequence of more general results, we get the  upper bound  $W_2 ( P_{S_n} , G_{n\sigma^2} ) \leq 6.825 \, \sigma^{-1} \sqrt{\BBE ( X_1^4)}$ ({\sl see} Remark \ref{ConstantsW2} for more details).
One can conjecture that the constant in  \eqref{RioBobkov} is less than $1$: most probably, to prove this conjecture, more sophisticated methods, such as the operator methods given in Bonis \cite{Bo24}, should be used.  Our techniques of proofs are more elementary. 
Therefore they can be adapted
to generalize the results of this paper to the class of weakly dependent sequences considered in Dedecker et al. \cite{DeMeRi23}.

\section{Main results} \label{MR}
\setcounter{equation}{0}

Throughout this paper, $(X_i)_{i \in {\mathbb Z}}$ is a sequence of iid real-valued random variables, 
$X$ denotes a real-valued random variable with the same law as $X_1$ and  $\Phi $ denotes the cumulative distribution function (c.d.f.) of a standard Normal distribution.

\subsection{Normal approximation for convex costs}

To handle the class of general convex costs we consider in this paper, we first introduce the following convex functions $\psi_x$, which generate all the convex functions in the class  $\Psi$ defined in Definition \ref{defclassPsi}  below (see Remark \ref{remark2}).  Some additional results are stated and proved in Appendix. 

For $x > 0$, define the convex and even function $\psi_x$ by
\begin{equation} \label{defpsi}
\psi_x (t)  = \left\{
\begin{aligned}
&t^2/4   & \text{if $ |t| \leq 2x$,}  \\  
&|tx| - x^2  &  \text{if $|t| > 2x$.} \\
  \end{aligned}
\right.
\end{equation}
The next proposition allows to compare the convex cost associated with $\psi_x$ with an appropriate ideal metric. 
\begin{Proposition}  \label{linkkappawasserstein} Let 
${\mathcal F}_x:= \{ f : {\mathbb R} \rightarrow {\mathbb R} \, : \, \Vert f' \Vert_{\infty}  \leq x \, , \, f' \text{ is $1$-Lipschitz} \}$.  Then
\[
\kappa_{\psi_x} ( \mu , \nu)  \leq \sup \bigl \{   \mu(f) -  \nu(f)  \, : \, f \in  {\mathcal F}_x \bigr \} .
\]
\end{Proposition}
Proposition \ref{linkkappawasserstein} and a Lindeberg-type method lead to the theorem below. 

\begin{Theorem} \label{Thmkappa} Let $(X_i)_{i \in {\mathbb Z}}$ be a sequence of iid real-valued random variables that are centered and in ${\mathbb L}^3$, with positive variance $\sigma^2$.  Let $\mu_3 = \BBE (X^3)$ and $\lambda_3 = \BBE (|X|^3)$. 
Then there exist  positive universal constants $\gamma_0$,
$\gamma_1$, $\gamma_2$ and $\gamma_3$  such that, setting  
$C_0 (X)= \gamma_1\sigma^2 +\sigma^{-4}(\gamma_0  \lambda_3^2 + \gamma_3 \mu_3^2)$,
for any $n \geq 1$ and any $x >0$, 
\[
\sqrt{\kappa_{\psi_{x} }( P_{S_n} , P_{G_{n\sigma^2}} )  }  \\ \leq  \sqrt{ C_0 (X)  +  \gamma_2 \sigma^{-2} \BBE  \big ( |X|^3 \min ( |X| ,  8 x ) \big ) }  
+ 0.968  \sigma^{-2} |\mu_3| \, ,
\]
where $G_{n\sigma^2} \sim {\mathcal N} ( 0, n \sigma^2)$.  The constants $\gamma_0$,  $\gamma_1$, $\gamma_2$ and $\gamma_3$ can be chosen as follows:
$\gamma_0 = 3/5$, $\gamma_1 =3.0376$, $\gamma_2 = 0.64584$ and $\gamma_3 = 1.6917$.
\end{Theorem} 
\begin{Remark} \label{ConstantsW2}
When the variables are in ${\mathbb L}^4$, taking the limit as $x$ tend to $\infty$ gives
\[
W_2 ( P_{S_n} , P_{G_{n\sigma^2}} )  \leq  2 \sqrt{C_0 (X) +  \gamma_2 \sigma^{-2} \mu_4 }  
+ 1.936     \sigma^{-2} |\mu_3| 
\]
for any $n\geq 1$, where $\mu_4 =   \BBE  (X^4 )$. In particular, since $\sigma^4 \leq \mu_4$ and 
$\mu_3^2 \leq \lambda_3^2 \leq \sigma^2 \mu_4$,
\[
W_2 ( P_{S_n} , P_{G_{n\sigma^2}} )  \leq 2 \sqrt{  4.2835 \sigma^{-2} \mu_4+    1.6917\sigma^{-4} \mu_3^2}  
+ 1.936  \sigma^{-2} | \mu_3|   \leq 6.825 \sigma^{-1}\sqrt{ \mu_4}   \, .
\]
In the case $\BBE (X^3)=0$,  the above upper bound yields 
\[
 W_2 ( P_{S_n} , P_{G_{n\sigma^2}} )  \leq    4.140 \sigma^{-1} \sqrt{ \mu_4 }   \, .
\]
Conversely, in the  case of Rademacher random variables with parameter $1/2$ (for which $\mu_3= 0$ and $\sigma^{-1} \sqrt{\mu_4} = 1$), \eqref{ERad} gives 
\[
\sup_{n \geq 1} W_2 ( P_{S_n} , P_{G_{n\sigma^2}} ) \geq   \sqrt{2 (1 - \sqrt{ 2/\pi } \,) }  \geq 0.63579 \, . 
\]
\end{Remark}

We now give applications of Theorem \ref{Thmkappa}  to a class of  convex costs whose definition is given below. 

\begin{Definition} \label{defclassPsi}

Let ${ \Psi} $ be the class of functions $\varphi$ defined on ${\mathbb R}$, even, convex, ${\cal C}^1$ and such that 
$ \varphi(0) =  \varphi' (0) =0$, $\varphi'$ is concave, $\varphi'$ is derivable at $0$, $\varphi''(0)=1 $ and $\lim_{x \rightarrow \infty} x^{-1}\varphi'(x) =0$. 
\end{Definition}
Note that the class $\Psi$ includes the function $x \mapsto  (|x|+1) \ln  (|x|+1) - |x|$. 
Next, for $p \in [1,2[$, let $g_p$ be the even function defined  on $ {\mathbb R}^+$ by 
\begin{equation} \label{defgp}
g_p (x)   = \left\{
\begin{aligned}
&x^2/2 & \text{if $ x \in [0,1]$}  \\  
& x^p/p + 1/2  - 1/p & \text{if $x >1$.} \\
  \end{aligned}
\right.
\end{equation}
It is easy to prove that $g_p$  is in the class ${\Psi}$, and that for any $x $ in ${\mathbb R}^+$,
\beq \label{comparisongpxp}
x^p \leq pg_p(x)  +1 - p/2 \, .
\eeq

\begin{Remark} \label{remark2}
Note that if $\varphi$ is in ${ \Psi} $ then there exists a probability measure $\nu$ on ${\mathbb R}_+$ such that 
\[ \varphi(z) = \frac{1}{2} \int_{ {\mathbb R}_+} \psi_t( 2 z ) d \nu (t) = 2 \int_{ {\mathbb R}_+} \psi_{t/2}(  z ) d \nu (t) 
\] where $\psi_t$ is defined in \eqref{defpsi}. This fact will be proved in Section \ref{demrem22}.
\end{Remark}
Let  $
{\mathcal  P}_{\varphi} ({\mathbb R} ) $ be the set of probability laws $\mu$ on the real line such that $\mu ( \varphi) < \infty$. 
For two probability laws  $\mu$ and $\nu$ in ${\mathcal  P}_{\varphi} ({\mathbb R} ) $,  define now 
\[
W_{\varphi} ( \mu, \nu ) = \bigl( \kappa_{\varphi} ( \mu, \nu ) \bigr )^{1/2} \, .
\]
$W_{\varphi}$ is a distance between probability laws on  ${\mathcal  P}_{\varphi} ({\mathbb R} )$ ({\sl see\/} Lemma \ref{Wphidistance}).

\begin{Theorem} \label{coravecphi} Let $(X_i)_{i \in {\mathbb Z}}$ be a sequence of iid real-valued random variables that are centered and in ${\mathbb L}^3$, with positive variance $\sigma^2$. Let $\varphi$ be a function 
in ${ \Psi} $ and $G_{n \sigma^2}$ be a  centered  normal  random variable with variance $n \sigma^2 $. Then, with the same notations
and the same constants as in Theorem \ref{Thmkappa}, for any $n \geq 1$, 
 \[
W_\varphi ( P_{S_n} , P_{G_{n\sigma^2}} )  \leq  \sqrt{ 2 C_0 (X) +  
 8\gamma_2 \sigma^{-2}  \BBE  \big ( |X|^3 \varphi' (  |X| /4  ) \big )  }  + 
1.369 \sigma^{-2}  |\mu_3|.
 \, ,
 \]
\end{Theorem} 
\begin{Remark} \label{remTh22} As we shall prove in Section \ref{section3.5}, 
\[
 \BBE  \big ( |X|^3 \varphi' (  |X| /4  ) \big )  
  \leq  8  \BBE  \big ( |X|^2 \varphi (  |X|  /4 ) \big ) \, .
\]
It follows that  $W_\varphi ( P_{S_n} , P_{G_{n\sigma^2}} ) $ is uniformly bounded in $n$ provided that 
\beq  \label{condmomentvarphi}
\BBE  \big ( |X|^2 \varphi (  |X|   ) )< \infty .
\eeq
\end{Remark}
Concerning the Wasserstein distances of order $p$, we obtain the  constants below in \eqref{RioBobkov}.
\begin{Proposition} \label{coravecxp}Let $p \in ]1,2[$ and $(X_i)_{i \in {\mathbb Z}}$ be a sequence of iid real-valued random variables that are centered and in ${\mathbb L}^{2+p}$. Then, for any $n \geq 1$, 
 \[
W_p ( P_{S_n} , P_{G_{n\sigma^2}} )    \leq  2 \bigl(   (C_0 (X))^{p/2}   + p 2^{5-3p} \gamma_2 \sigma^{-2} \BBE  ( |X|^{p+2})  \bigr )^{1/p}
+   1.936 \sigma^{-2}  |\mu_3 | .
\]
\end{Proposition} 

Starting from Theorem \ref{coravecphi} and using the Cornish-Fisher expansion, one obtains the following extension
of the results of Rio \cite{Rio11} on asymptotic constants 
for general convex costs. 

\begin{Corollary} \label{asymptoticconstants} Let $(X_i)_{i \in {\mathbb Z}}$ be a sequence of iid real-valued random variables that are centered and in ${\mathbb L}^3$, with  variance $1$. Let $G_n$ be a  centered  normal  random variable with variance $n$
and $G$ be a random variable with standard normal law. 
Let $\varphi$ be any function in ${ \Psi} $ such that $\BBE  \big ( |X|^2 \varphi (  |X|   ) )< \infty$. 
\par\smallskip\no
{\bf (a)} If the distribution of $X$ is not a lattice distribution, then  
$$
\lim_{n\rightarrow \infty} \kappa_\varphi ( P_{S_n} , P_{G_n} ) = \BBE \bigl ( \varphi \bigl( \BBE ( X^3) (G^2- 1)/6 \bigr) \bigr) . 
$$
\par\smallskip\no
{\bf (b)} If $X$  takes its values 
in the arithmetic progression $\{ a+kh : k\in\mathbb Z \}$ ($h$ being maximal) and $V$ is a random variable with  uniform law over 
$[-1/2 , 1/2]$, independent of $(X_i)_{i \in {\mathbb Z}}$, then 
$$
\lim_{n\rightarrow \infty} \kappa_\varphi ( P_{S_n +hV} , P_{G_n} ) = \BBE \bigl ( \varphi \bigl( \BBE ( X^3) (G^2- 1)/6 \bigr) \bigr) . 
$$
\end{Corollary}
\smallskip
\begin{Remark} \label{remarkasymptoticconstants} Note that the above limits are equal to $0$ if and only if $\BBE (X^3)=0$. In addition, as in Rio \cite{Rio11}, 
one can also prove that if $X$  takes its values 
in the arithmetic progression $\{ a+kh : k\in\mathbb Z \}$ ($h$ being maximal) and $U$ is a random variable with  uniform law over $[0,1]$, independent
of $G$, then  
$$
\lim_{n\rightarrow \infty} \kappa_\varphi ( P_{S_n} , P_{G_n} ) = \BBE \bigl ( \varphi \bigl( \BBE ( X^3) (G^2- 1)/6 + h(U - 1/2) \bigr) \bigr) . 
$$
\end{Remark}
\smallskip
Let $U$ denote a random variable with uniform law over $[0,1]$. From \eqref{Major78}, 
\beq \label{expressionkappaphi}
\kappa_\varphi (P_{S_n} , P_{G_n} ) = \BBE ( \varphi ( Z_n ) ), \text{ with } Z_n = | F_{S_n}^{-1} (U) - \sqrt{n} \Phi^{-1} (U) | . 
\eeq
The proof of Corollary \ref{asymptoticconstants} is based on the uniform integrability of the sequence $(\varphi (Z_n))_{n\geq 1}$,
proved in Lemma \ref{lmaUI}. Using this lemma, one can also obtain asymptotic constants for the weighted costs defined below.
Let $g : [0,1] \mapsto \BBR$ be a positive and measurable function. Define the weighted cost $\kappa_{1,g}$ by 
\beq \label{defkappa1p}
\kappa_{1,g} ( P_X, P_Y ) = \BBE \bigl( g(U) |F_X^{-1} (U) - F_Y^{-1} (U) | \bigr) .
\eeq
\par
For $\varphi$ in $\Psi$, the Young dual $\varphi^*$ of $\varphi$ is defined by
$\varphi^* (x) = \sup_{t\in \BBR} (xt - \varphi (t))$.
We denote by $L_{\varphi^*}$  the space of real-valued random variables $X$ such that $\BBE ( \varphi^* ( aX)  ) < \infty$
for some positive $a$. The corollary below provides asymptotic constants for the costs $\kappa_{1,g}$ when $X$
satisfies the integrability condition of Corollary \ref{asymptoticconstants} and  $g(U)$ belongs to the space $L_{\varphi^*}$. 

\begin{Corollary} \label{asymptotickappa1g} Let $(X_i)_{i \in {\mathbb Z}}$ be a sequence of iid real-valued random variables that are centered and in ${\mathbb L}^3$, with  variance $1$. Let $G_n$ be a  centered  normal  random variable with variance $n$
and $G$ be a random variable with standard normal distribution and $U = \Phi (G)$. 
Let $\varphi$ be any function in $\Psi $ such that $\BBE  \big ( |X|^2 \varphi (  |X|   ) )< \infty$
and $g : ]0,1[ \mapsto \BBR$ be a positive and measurable function such that $g(U)$ belongs to $L_{\varphi^*}$. 
\par\smallskip\no
{\bf (a)} If the distribution of $X$ is not a lattice distribution, then  
$$
\lim_{n\rightarrow \infty} \kappa_{1,g} ( P_{S_n} , P_{G_n} ) = | \BBE (X^3)| \, \BBE \bigl( g(\Phi (G)) |G^2 - 1 |/6 \bigr). 
$$
\par\smallskip\no
{\bf (b)} If $X$  takes its values 
in the arithmetic progression $\{ a+kh : k\in\mathbb Z \}$ ($h$ being maximal) and $V$ is a random variable with  uniform law over 
$[-1/2 , 1/2]$, independent of $(X_i)_{i \in {\mathbb Z}}$, then 
$$
\lim_{n\rightarrow \infty} \kappa_{1,g} ( P_{S_n +hV} , P_{G_n} ) = | \BBE (X^3)| \, \BBE \bigl( g(\Phi (G)) |G^2 - 1 |/6 \bigr). 
$$
\end{Corollary}
\smallskip
let us illustrate Corollary \ref{asymptotickappa1g} when $\varphi$ is defined by $\varphi (x) = (1+x) \log (1+x) - x$ for $x\geq 0$. Since 
$\varphi^* (x) = e^x-1-x$ for $x\geq 0$, the random variable
$(\Phi^{-1} (U) )^2$ belongs to $L_{\varphi^*}$. Consequently, in the non lattice case, if 
$\BBE (|X|^3 \log ( 1+ |X| ) )< \infty$, 
\[
\lim_{n\rightarrow \infty} \int_0^1 (\Phi^{-1} (u))^2 | F_{S_n}^{-1} (u) - \sqrt{n} \Phi^{-1} (u) | du = 
| \BBE (X^3)| \, \BBE \bigl( G^2 |G^2 - 1 |/6 \bigr). 
\]
Moreover $\BBE \bigl( G^2 |G^2 - 1 |/6 \bigr)= 2 (2\pi e)^{-1/2} + 1 - (4/3) \Phi (1)  \simeq 0.3622$. 
\smallskip
\begin{Remark} \label{signedfunctionals}  Using the same line of proof, one can easily extend the results of 
Corollary \ref{asymptotickappa1g} to signed functionals as follows. Let $\ell : \BBR \mapsto \BBR$ be a Lipschitz function 
and $g : ]0,1[ \mapsto \BBR$ be a measurable function such that $g(U)$ belongs to $L_{\varphi^*}$. Then, in the non lattice case
\beq \label{asymptoticsignedfunctionals} 
\lim_{n\rightarrow \infty} \int_0^1 g(u) \ell  \bigl( F_{S_n}^{-1} (u) - \sqrt{n} \Phi^{-1} (u) \bigr) du = 
\BBE \bigl(  g(\Phi (G)) \ell \bigl(  \BBE (X^3) (G^2 - 1 )/6 \bigr) \, \bigr)
\eeq
In the lattice case the same result holds with $F_{S_n+hV}^{-1} (u)$ instead of $F_{S_n}^{-1} (u)$ on left hand. 
Applying \eqref{asymptoticsignedfunctionals} with $\ell (x) = x$ and $g = ( \Phi^{-1} )^2$, one obtains that, 
in the non lattice case, if  $\BBE (|X|^3 \log ( 1+ |X| ) )< \infty$, 
\[
\lim_{n\rightarrow \infty} \int_0^1 (\Phi^{-1} (u))^2  ( F_{S_n}^{-1} (u) - \sqrt{n} \Phi^{-1} (u) ) du = 
\BBE (X^3) \, \BBE \bigl( G^2 (G^2 - 1 )/6 \bigr) = \BBE (X^3) /3. 
\]

\end{Remark}
\subsection{Coupling inequalities}

Enlarging the probability space if necessary,  recall that, for any $n \geq 1$, 
\[
\kappa_p ( P_{S_n} , P_{G_{n\sigma^2}} )   = \BBE \big (  \big |  S_n - G_{n\sigma^2} \big |^p\big ) 
\]
where the random variable $G_{n \sigma^2}$ is a centered normal random variable (r.v.) with variance $n \sigma^2$,  defined from $S_n$ via the quantile transformation as follows: 
\beq \label{defGn}
G_{n \sigma^2}  =  \sigma \sqrt{n} \Phi^{-1}  \bigl( F_n ( S_n - 0) + \delta ( F_n (  S_n) -F_n ( S_n- 0)  ) \bigr )  , 
\eeq
where $F_n$ is the c.d.f. of $S_n$,   $\delta$ is a r.v. with uniform distribution over $[0,1]$ independent of $(X_i)_{i \in {\mathbb Z}}$ and $F_n(x-0)$ means the left limit of $F_n$ at $x$.  

Recall that Inequality \eqref{RioBobkov} ensures that if the  random variables have a strong  moment of order $2 + p$ with $p\in ]1,2]$,  $\kappa_p ( P_{S_n} , P_{G_{n\sigma^2}} )  $ is uniformly bounded with respect to $n$.  Now, starting from the tools introduced in the proofs of the above 
results, we give upper bounds on the weak norm of order $p$ of the random variable $Z_n$ defined by 
\beq \label{defZn}
Z_n = |S_n - G_{n\sigma^2} |  \, , 
\eeq
where $G_{n \sigma^2}$ is defined by \eqref{defGn}, under suitable conditions on the tail of the random variables $X_k$. In order to get more general results, we will introduce the conditional value at risk of a real valued random variable and its generalized inverse. For a real-valued random variable $Z$ we define its conditional value at risk by: for any $u \in ]0,1[$
\beq \label{defCVAR}
\text{CVar}(Z) (u) =  \inf \bigl \{ t+ u^{-1}  \BBE (  (Z-t)_+ ) : t>0  \bigr \} \, .
\eeq
\par\smallskip
According to Proposition 5.4 in Pinelis \cite{Pi2014}, 
\[
\text{CVar}(Z) (u) = {\tilde Q}_Z (u) = \frac{1}{u} \int_0^u Q_Z(v) dv \, , 
\]
where $Q_Z$ is the tail quantile function of $Z$, i.e. $Q_Z (v) = F_Z^{-1} (1-v)$.  Note that ${\tilde Q}_Z$ was introduced by Hardy and Littlewood \cite{HL1930}.  
We also define  $\tilde H_Z$  by 
\[
\tilde H_Z (x) = \inf \{ (x-t)^{-1} \BBE ( (Z-t)_+ ) : t<x \}
\]
Define also the tail function $H_Z$ by $H_Z (x) = \BBP ( Z>x)$. By the Markov inequality $\tilde H_Z \geq H_Z $. 
From the variational formula \eqref{defCVAR}, $\tilde H_Z$ is the generalized inverse of $\tilde Q_Z$. Indeed,
for any $u$ in $]0,1]$, 
\beq \label{GeneralizedInverse}
\tilde H_Z (x) < u \text{ if and only if } x> \tilde Q_Z (u) .
\eeq
It follows that $\tilde H_Z$ is the tail function of  $\tilde Q_Z$. Hence an upper bound on $\tilde H_Z$ provides immediately an upper bound on $\tilde Q_Z$. Below we give an upper bound on $\tilde H_{Z_n}$, which will allow us to extend \eqref{RioBobkov} to weak moments.

\begin{Theorem} \label{UpperBoundHLtail} Under the conditions of Theorem \ref{Thmkappa} and with the same notations, for any positive $t$, 
\[
t^2 \tilde H_{Z_n} (t) \leq 42.943 \, \sigma^{-4} \lambda_3^2 + 5.2041  \, \sigma^{-2} \BBE \bigl( |X|^3 \min ( |X| , 2.8183 t ) \bigr) .
\]
\end{Theorem}

From Theorem \ref{UpperBoundHLtail}, we now derive upper bounds  on the tail of $Z_n$
in case of weak moment of order $2 + p$ with $p\in ]1,2[$. 
For any $q \geq 1$ and a real-valued random variable $Z$, define the following weak moments of order $q$ of $Z$: 
\beq \label{weakmomentsq}
\Lambda_q(Z) =  \sup_{x >0} x^{q } H_{|Z|} (x)  \quad \text{and} \quad
\tilde \Lambda_q (Z) =\sup_{x >0} x^{q } \tilde H_{|Z|} (x) \, .
\eeq
Clearly $\Lambda_q (Z) \leq \tilde\Lambda_q (Z)$. Starting from Theorem \ref{UpperBoundHLtail}, we obtain the following upper bound
on the weak moments of $Z_n$.

\begin{Corollary} \label{corweakmoments} Let $(X_i)_{i \in {\mathbb Z}}$ be a sequence of iid real-valued random variables that are centered, with variance $\sigma^2$ and  such that $ \Lambda_{p+2} (X) < \infty$ for some $p \in ]1,2[$. Assume that  the probability space is large enough and, for any $n \geq 1$,  define $G_{n \sigma^2}$ by \eqref{defGn} and let 
$Z_n = | S_n -  G_{n \sigma^2}  |$. Then  
\[
\tilde\Lambda_p (Z_n) \leq  (a_1 \, \sigma^{-2} \lambda_3 )^p +  a_2 \bigl( (p-1) (2-p)\bigr)^{-1} (p+2)  \sigma^{-2} \Lambda_{p+2} (X) ,
\] 
with $\lambda_3 = \BBE (|X|^3)$, $a_1 = 6.5531$ and $a_2 = 5.2041 (2.8183)^{2-p}$.  In addition, if $\BBE (|X|^{p+2}) < \infty$, 
\[
\tilde\Lambda_p (Z_n) \leq  (a_1 \, \sigma^{-2} \lambda_3 )^p +  a_2 \sigma^{-2} \BBE (|X|^{p+2})  .
\]
\end{Corollary} 

Now, define the Calderon weak norm of order $p$ (see DeVore and Lorentz \cite{DL93}, page 26) of a random variable $Z$ by 
\[
\Vert Z\Vert_{w , p} = \sup_{u\in ]0,1]} u^{1/p} \tilde Q_{|Z|} (u) . 
\]
One can easily deduce from \eqref{GeneralizedInverse} that 
\[
\tilde\Lambda_p ( Z_n) = \Vert Z_n \Vert_{w,p}^p .
\]
 Hence  Corollary \ref{corweakmoments} provides immediately the upper bound below:
\[ 
\Vert Z_n \Vert_{w , p}  \leq \tilde \kappa_p^{1/p} \  \text{ with } \  \tilde \kappa_p = (a_1\sigma^{-2} \lambda_3)^p + a_2 \bigl(
 (p-1) (2-p)\bigr)^{-1} (p+2) \sigma^{-2} \Lambda_{p+2} (X)  .
\]

\begin{Remark} \label{remarkcorcvar}
By subadditivity and monotonicity of the conditional value at risk (see Theorem 3.4 in  \cite{Pi2014}), one gets that for  real-valued random variables $X$ and $Y$, 
\[ \big |  {\tilde Q}_X (u)  -  {\tilde Q}_Y (u)  \big |  \leq  {\tilde Q}_{|X-Y|} (u)  \, .
\]
Hence, under the conditions of Corollary \ref{corweakmoments}, for any $n \geq 1$ and any $u \in ]0,1[$,  
\beq \label{ine1condquant}
 \big | {\tilde Q}_{S_n} (u)  -   {\tilde Q}_{G_{n\sigma^2}} (u)  \big |  \leq   \bigl( \tilde \kappa_p/u \bigr)^{1/p}     \, .
\eeq
Consequently
\[
{\tilde Q}_{S_n} (u)  \leq  \Bigl( \frac{{\tilde \kappa}_p }{u} \Bigr)^{1/p}  + \frac{\sigma \sqrt{n}}{u \sqrt{2 \pi}} \exp \Big (- \frac{ (\Phi^{-1} (u) )^2}{2} \Big )   \, .
\] 
Note that, if the random variables admit a strong moment of order $p+2$ for $p \in ]1,2] $, the upper bound \eqref{ine1condquant} can be improved. Indeed, Corollary 3.1 in Rio \cite{Rio2017CRAS}  
asserts that in this case there exists a positive constant $C_p$ such that for any $u$ in $]0,1/2]$,
\[
\big | {\tilde Q}_{S_n} (u)  -   {\tilde Q}_{G_{n\sigma^2}} (u)  \big |  \leq C_p  ( u  |\ln (u) |) ^{-1/p} \, . 
\]
\end{Remark}

We now give asymptotic constants for $\Lambda_p ( Z_n) $ and ${\tilde \Lambda}_p ( Z_n) $ under an additional condition on the tail of $|X|$.

\begin{Corollary} \label{asymptoticconstantsforLambda} Let $p \in [1,2[$ and $(X_i)_{i \in {\mathbb Z}}$ be a sequence of iid real-valued random variables in ${\mathbb L}^3$ that are centered, with variance $1$. If $p >1$, assume furthermore that 
\beq \label{conditionontail}
\lim_{t \rightarrow + \infty} t^{p+2} {\mathbb P} ( |X| > t ) =0 \, .
\eeq
\par\smallskip\no
{\bf (a)} Assume that the distribution of $X$ is not a lattice distribution. Then  
$$
\lim_{n\rightarrow \infty} \Lambda^{1/p}_p ( Z_n)= \frac{| \BBE ( X^3)|}{6}  \Lambda^{1/p}_p (G^2- 1) \text{ and } \lim_{n\rightarrow \infty} {\tilde \Lambda}^{1/p}_p ( Z_n)= \frac{| \BBE ( X^3)|}{6}  {\tilde \Lambda}^{1/p}_p (G^2- 1) .
$$
\par\smallskip\no
{\bf (b)} Assume that $X$  takes its values 
in the arithmetic progression $\{ a+kh : k\in\mathbb Z \}$ ($h$ being maximal). Let  $V$ be a random variable with  uniform law over 
$[-1/2 , 1/2]$, independent of $(X_i)_{i \in {\mathbb Z}}$ and ${\bar S}_n = S_n +hV$. Define ${\bar G}_n = \sqrt{n} \Phi^{-1} (F_{{\bar S}_n } ( {\bar S}_n ))$ and 
${\bar Z}_n = |{\bar S}_n - {\bar G}_n|$. Then
$$
\lim_{n\rightarrow \infty} \Lambda^{1/p}_p ( {\bar Z}_n)= \frac{| \BBE ( X^3)|}{6}  \Lambda^{1/p}_p (G^2- 1) \text{ and }
 \lim_{n\rightarrow \infty} {\tilde \Lambda}^{1/p}_p ( {\bar Z}_n)= \frac{| \BBE ( X^3)|}{6}  {\tilde \Lambda}^{1/p}_p (G^2- 1) .
$$
\end{Corollary}
\begin{Remark}
When $ \BBE ( X^3)=0$, the asymptotic constants are equal to zero. 
\end{Remark}

\section{Proofs} \label{sectionPR}

\setcounter{equation}{0}

\subsection{Preliminary considerations} \label{sectionPC}

We start with a smoothing lemma which is a consequence of Lemma 6.1 in \cite{DMR09}.  Below ${\mathcal F}_x$ is the class of functions defined in Proposition \ref{linkkappawasserstein}. 

\begin{Lemma}  \label{smoothlemma} Let $f$ in ${\mathcal F}_x$, $N$ be a standard normal and $B$ be an integrable real-valued random variable independent of $N$. Let $c_{j} = \Vert \phi^{(j)} \Vert_1$ where $\phi$ denotes the density of $N$. Then, for any integer $i \geq 3$ and any $t >0$, 
\[
 \Big | \frac{d^i}{d u^i} \BBE f ( u +  t N + B ) \Big |  \leq   \min \Big ( c_{i-1}  \, x  t ^{1-i} ,  c_{i-2}  \, t^{2-i} \Big ) .
\]

\end{Lemma}
Elementary computations lead to 
\[
c_1  =\sqrt{ 2/\pi}  \, , \, c_2 = \sqrt{ 8/ (\pi {\rm e})}  \leq 0.9679 \, , \,
c_3 =  \sqrt{ 2/\pi }  \bigl  ( 1+ 4 {\rm e}^{-3/2}  \bigr)  \, . 
\]
In addition 
\[
c_4 =  4 \sqrt{ 3 / ( \pi {\rm e}^3) }\bigl (   ( 3 - \sqrt{6})^{1/2} {\rm e}^{ \sqrt{3/2}} +   ( 3 + \sqrt{6})^{1/2} {\rm e}^{- \sqrt{3/2}}   \bigr ) 
 \, , 
\]
\[
  c_5 = \frac{4}{ \sqrt{2 \pi}} \Big (  \frac32  + 4 ( 2 + \sqrt{10}) {\rm e}^{- (5 + \sqrt{10})/2}  + 4 (  \sqrt{10} - 2 ) {\rm e}^{- (5 - \sqrt{10})/2}  \Big )  \leq 5.9101  \, .
\]

\medskip

Recall that the class $\Psi$ has been defined in Definition \ref{defclassPsi}.

\begin{Lemma}  \label{Wphidistance}
Let $\varphi $ in $\Psi$. Then $W_{\varphi}$ is a distance  on  the set ${\mathcal  P}_{\varphi} ({\mathbb R} )$ of probability laws $\mu$ on the real line such that $\mu ( \varphi) < \infty$. 
\end{Lemma}
\noindent {\bf Proof of Lemma \ref{Wphidistance}.} Let $\mu_1$, $\mu_2$ and  $\mu_3$ in  ${\mathcal  P}_{\varphi} ({\mathbb R} ) $ with respective distribution functions $F_1$, $F_2$ and $F_3$. 

Clearly 
$W_{\varphi} ( \mu_1, \mu_2 )  \geq 0$ and $W_{\varphi} ( \mu_1, \mu_2 )=0$  iff $\mu_1 = \mu_2$. We now prove the triangular inequality.  Let $U \sim {\mathcal U} ( [0,1] ) $. Set $A = F_1^{-1} (U) -   F_2^{-1} (U)$ and $B = F_2^{-1} (U) -   F_3^{-1} (U)$. With these notations, we have
\[
\kappa_{\varphi} ( \mu_1, \mu_3) = \BBE \big ( \varphi ( A+B) \big )  \, , \,  \kappa_{\varphi} ( \mu_1, \mu_2) = \BBE \big ( \varphi ( A) \big )   \, , \,  \kappa_{\varphi} ( \mu_2, \mu_3) = \BBE \big ( \varphi ( B) \big )  \, .
\]
Hence proving the triangular inequality is equivalent to prove that 
\beq \label{TI}
\sqrt{\BBE \big ( \varphi ( A+B) \big )} \leq  \sqrt{\BBE \big ( \varphi ( A) \big )}  + \sqrt{\BBE \big ( \varphi ( B) \big )}  \, .
\eeq
To prove \eqref{TI}, we first note that $\sqrt{\varphi}$ is subadditive on ${\mathbb R}_+$.  Indeed,   since $\varphi'$ is concave and  $\varphi' (0) =0$, $x \mapsto \frac{\varphi' (x) }{x} $ is nonincreasing on $]0, \infty[$ and so 
$x \mapsto \frac{\varphi (x) }{x^2} $ also since $\varphi(0) =0$, by l'Hospital rule for monotonicity (see Corollary 1.3 and Remark 1.2 in Pinelis \cite{Pi}). Hence, since $x \mapsto x^{-1} \sqrt{\varphi (x) }$ is nonincreasing on $]0, \infty[$, for any positive reals $a$ and $b$, 
\[
\frac{ \sqrt{\varphi (a+b)} }{a+b}  \leq \min \Big (  \frac{ \sqrt{\varphi (a)} }{a}  , \frac{ \sqrt{\varphi (b)} }{b} \Big )  \leq 
\frac{a}{a+b} \frac{ \sqrt{\varphi (a)} }{a}  + \frac{b}{a+b}  \frac{ \sqrt{\varphi (b)} }{b}  \, .
\]
This implies that, for any positive reals $a$ and $b$, 
\beq  \label{subadvarphi}
\sqrt{\varphi (a+b)}  \leq \sqrt{\varphi (a)}  + \sqrt{\varphi (b)} \, .
\eeq
We go back to the proof of \eqref{TI}. Using  \eqref{subadvarphi}, we get 
\[
\BBE \big ( \varphi ( A+B) \big )  \leq  \BBE \big (  \big ( \sqrt{ \varphi (  A) }  + \sqrt{ \varphi (  B) }  \big )^2\big ) \, .
\]
Hence, by the Schwarz inequality, 
\begin{multline*}
\BBE \big ( \varphi ( A+B) \big )  \leq \BBE \big ( \varphi ( A) \big )  + \BBE \big ( \varphi ( B) \big )   + 2   \BBE \big ( \sqrt{ \varphi (  A) } \sqrt{ \varphi (  B) }  \, \big ) \\
\leq  \Big (   \sqrt{\BBE \big ( \varphi ( A) \big )  } + \sqrt{ \BBE \big ( \varphi ( B) \big )  } \Big )^2 \, , 
\end{multline*}
which ends the proof of \eqref{TI}.  \qed

\medskip

The next result gives a quantitative upper bound for the Wasserstein distance of order $2$  between the Poisson distribution and the  normal distribution with the same mean and same variance. The proof  will be given in Appendix.

\begin{Proposition} \label{PropPoisson} Let $(\Pi (t))_{t \geq 0}$ be a Poisson Process with parameter $1$. For any $m >0$ and any $\alpha \neq 0$,  denote by 
$\mu_{m, \alpha} $  the law of $ \alpha (  \Pi ( \alpha^{-2} m ) -  \alpha^{-2} m) $ and $\nu_m$ the ${\mathcal N} ( 0,m)$ law.  Then
$W_2^2 \big ( \mu_{m, \alpha} , \nu_m \big )  \leq  0.937 \, \alpha^2$.
\end{Proposition}

\subsection{Proof of Theorem \ref{Thmkappa}}

Let $\sigma^2 = \BBE (X^2) >0$. Note that, for any positive $a$, 
\beq \label{inepsi}
\psi_{ax} ( az ) = a^2  \psi_{x } (z) \, . 
\eeq
This leads to 
\beq \label{toassumesigma1}
\kappa_{\psi_x} ( P_{A} , P_{B} ) = a^2 \kappa_{\psi_{x/a}} ( P_{a^{-1}A} , P_{a^{-1}B}  ) \, .
\eeq
Therefore we shall prove the result in case $\sigma^2=1$ and use the above equality with $A=S_n$, $B= G_{n\sigma^2}$ and 
$a=\sigma$ to get the general case. So, from now, we assume that  $\sigma^2=1$. 

Let $\beta _3 = \BBE ( X^3)$. Assume first that $\beta _3 \neq 0$.

Let $(N_k)_{k \geq 1}$ be iid Gaussian random variables that are independent of $(X_k)_{k \geq 1}$, centered and with variance $1/2$. Let 
$(B_k)_{k\geq 1}$ be a sequence of iid random variables, independent of $(N_k,X_k)_{k \geq 1}$ and with the same law as 
\[
B = 2 \beta_3 \bigl ( \Pi \bigl (  1/(8 \beta_3^2) \bigr )   -   1/(8 \beta_3^2) \bigr ) \, , 
\]
where  $\Pi(\lambda)$ is a r.v. with Poisson distribution of parameter $\lambda$.  For any $k \geq 1$, let $Y_k = B_k + N_k$.  
The random variables $Y_k$ are iid  and their first moments satisfy
\beq \label{defmomentsY}
\BBE ( Y_1) =0 , \BBE ( Y_1^2) =1 ,  \BBE ( Y_1^3) = \beta_3 
\text{ 
and }
\BBE ( Y_1^4) = 3 + 2 \beta_3^2  . 
\eeq
We now define set $T_n = \sum_{k=1}^n Y_k$. Note now that $2 \psi_x$ belongs to $\Psi$. Hence, by Lemma \ref{Wphidistance},  $\sqrt{\kappa_{\psi_x} }$ is a distance  between probability laws, from which
\beq \label{ineconvexity}
\sqrt{\kappa_{\psi_{x}} ( P_{S_n} , P_{G_n }) } \leq \sqrt{  \kappa_{\psi_{x}} ( P_{S_n} , P_{T_n} )  }+ 
 \sqrt{  \kappa_{\psi_{x}} ( P_{T_n} , P_{G_{n} } )}  \, . 
\eeq
Next, using the fact that $\psi_x (t) \leq t^2/4$ and Proposition \ref{PropPoisson} with $\alpha = 2 \beta_3$ and $m=n/2$, we get 
\beq \label{boundfinal1}
\kappa_{\psi_x} ( P_{T_n} ,  P_{G_{n} }) \leq  \frac{1}{4} W_2^2  ( P_{T_n} ,  P_{G_{n} } ) \leq  \frac{1}{4} W_2^2 (  \mu_{n/2,  2 \beta_3} ,  P_{G_{n/2} }) \leq 0.937 \beta_3^2 \, .
\eeq
The above inequalities together with Proposition  \ref{linkkappawasserstein} imply that 
\beq \label{boundfinal2}
\sqrt{\kappa_{\psi_{x}} ( P_{S_n} , P_{G_n }) } \leq \sqrt{ \sup \bigl \{   \BBE ( f(S_n) - f(T_n) )  \, : \, f \in  {\mathcal F}_x \bigr \} } + 0.968 |\beta_3|.
\eeq
It remains to give an  upper bound on $\BBE ( f(S_n) - f(T_n) )$ for  $f$ in ${\mathcal F}_x$.
Let $v^2 >1/2$ to be chosen later  and $Z$ be a random variable with normal distribution ${\mathcal N} (0,v^2)$ such that $Z$ is independent of $(X_k,Y_k)_{k \in {\mathbb Z}}$. We first write
\beq \label{ine1}
\BBE ( f(S_n))-  f(T_n)) = \BBE ( f(S_n +Z)) - f(T_n + Z)) + R_{n,2} - R_{n,1} ,
\eeq
with  $R_{n,1} = \BBE ( f (S_n+Z) - f(S_n) )$  and $R_{n,2} = \BBE ( f (T_n+Z) - f(T_n) )$.
Now, since $Z$ is independent of $(S_n,T_n)$, $\BBE ( f' (S_n) Z ) = \BBE ( f' (T_n) Z ) = 0$. It follows that
$R_{n,1} = \BBE ( f (S_n+Z) - f(S_n) - f' (S_n) Z )$, which together with the fact that $f'$ is $1$-Lipschitz implies that 
$|R_{n,1}|  \leq \BBE ( Z^2/2) = v^2/2$. Similarly $|R_{n,2}|  \leq v^2/2$, whence
\beq \label{ineforRn}
|R_{n,2}- R_{n,1}| \leq  v^2 \, .
\eeq
From now on, let $\Delta_{n,v} (f) = \BBE ( f(S_n +Z)) - f(T_n + Z))$. Setting 
\[
f_{n-k}(x)  =  \BBE \Big (  f \big ( x + \sum_{u=k+1}^n Y_u + Z  \big )  \Big ) \, , 
\]
and taking into account the independence between the sequences, we get that
\begin{multline*}
 \Delta_{n,v} (f)  = \sum_{k=1}^n   \BBE ( f(S_{k-1} +X_k + T_n - T_k +Z)) -  f(S_{k-1} +Y_k + T_n - T_k +Z))     \\
 =  \sum_{k=1}^n   \Big (  \BBE ( f'_{n-k}(S_{k-1} ) (X_k -Y_k ))  + \frac{1}{2}  \BBE ( f''_{n-k}(S_{k-1} ) (X^2_k -Y^2_k ))  \Big )    \\
 +  \int_0^1 \frac{(1-s)^2}{2} \sum_{k=1}^n  \big \{  \BBE  ( f^{(3)}_{n-k} (S_{k-1} + s X_k )  X_k^3 ) -  \BBE  ( f^{(3)}_{n-k} (S_{k-1} + s Y_k )  Y_k^3 )     \big \}  ds \, .
\end{multline*}
By independence and since $ \BBE (X_k) =  \BBE (Y_k)$ and $ \BBE (X^2_k) =  \BBE (Y^2_k)$, it follows that 
\beq \label{ine2}
 \Delta_{n,v} (f) 
 =  \int_0^1 \frac{(1-s)^2}{2} \sum_{k=1}^n  \big \{  \BBE  ( f^{(3)}_{n-k} (S_{k-1} + s X_k )  X_k^3 ) -  \BBE  ( f^{(3)}_{n-k} (S_{k-1} + s Y_k )  Y_k^3 )    \big \}  ds \, .
\eeq
Let us first handle the quantity
\[
A (n):=  \int_0^1 \frac{(1-s)^2}{2} \sum_{k=1}^n  \big \{  \BBE  ( f^{(3)}_{n-k} (T_{k-1} + s X_k )  X_k^3 ) -  \BBE  ( f^{(3)}_{n-k} (T_{k-1} + s Y_k )  Y_k^3 )     \big \} ds \, .
\]
By independence and the fact that $\BBE ( X_k^3) = \BBE ( Y_k^3)= \beta_3$, we have
\[
\BBE  ( f^{(3)}_{n-k} (T_{k-1}  )  X_k^3 )   = \BBE  ( f^{(3)}_{n-k} (T_{k-1}  )   )  \beta_3    = \BBE  ( f^{(3)}_{n-k} (T_{k-1}  )  Y_k^3 )  \, .
\]
It follows that $A (n) = A_1(n) - A_2 (n)$,
where 
\[
A _1(n):=  \int_0^1 \frac{(1-s)^2}{2} \sum_{k=1}^n  \big \{  \BBE  ( f^{(3)}_{n-k} (T_{k-1} + s X_k )  X_k^3 ) -  \BBE  ( f^{(3)}_{n-k} (T_{k-1}  )  X_k^3 )     \big \}  ds
\]
\[
A _2(n):=  \int_0^1 \frac{(1-s)^2}{2} \sum_{k=1}^n  \big \{  \BBE  ( f^{(3)}_{n-k} (T_{k-1} + s Y_k )  Y_k^3 ) -  \BBE  ( f^{(3)}_{n-k} (T_{k-1}  )  Y_k^3 )     \big \}  ds .
\]
\par
Let $u_n := (n-1)/2 + v^2$. Since $v^2 \geq 1/2$, we have that $u_n \geq n/2$. 
For any $s \in [0,1]$, taking into account Lemma \ref{smoothlemma}, 
\begin{multline*}
\sum_{k=1}^n  \BBE   \Bigl (  \big |  \bigl( f^{(3)}_{n-k} (T_{k-1} + s X_k )  -  f^{(3)}_{n-k} (T_{k-1}  )  \bigr)  X_k^3 \big|     \Bigr ) \\
\leq n  \BBE  \Big ( |X|^3   \min \Big ( 2 \Vert   f^{(3)}_{n} \Vert_{\infty} , s |X|     \Vert   f^{(4)}_{n} \Vert_{\infty}  \Big )   \Big ) 
\leq n   c_2 \BBE  \Big ( |X|^3   \min \Big \{  \frac{2x}{u_n}     ,   |X|    \frac{s}{u_n}   \Big \}    \Big )  \, .
\end{multline*}
Since $(n/u_n) \leq 2$, it follows that 
\beq \label{BoundA1}
|A _1(n)| \leq  \frac{c_2}{12}   \BBE   ( |X|^3   ( |X|    \wedge 8x )   ) \, .
\eeq
On another hand, for any $s$ in $[0,1]$, taking into account Lemma \ref{smoothlemma}, 
\[
\sum_{k=1}^n  \BBE   \Bigl (  \big |  \bigl( f^{(3)}_{n-k} (T_{k-1} + s Y_k )  -  f^{(3)}_{n-k} (T_{k-1}  ) \bigr)Y_k^3  \big |   \Bigr ) 
\leq n s    \Vert   f^{(4)}_{n} \Vert_{\infty}   \BBE  (  Y^4    )  \leq c_2 \frac{n s }{u_n}    ( 3 + 2 \beta_3^2)   \, .
\]
Therefore
\beq \label{BoundA2}
|A _2(n)| \leq   \frac{c_2}{12} \times    ( 3 + 2 \beta_3^2) \, .
\eeq
Taking into account \eqref{BoundA1} and \eqref{BoundA2}, we finally get that
\beq \label{BoundAn}
|A (n)| \leq   \frac{c_2}{12}  \Big (3 + 2 \beta_3^2 +  \BBE   ( |X|^3   ( |X|    \wedge  8 x )   ) \Big )  \, .
\eeq
Next, we write 
\beq \label{decwithAn}
\Delta_{n,v} (f)  - A (n) =: B_1 (n) -  B_2 (n)   \, ,
\eeq
where
\[
B_1 (n) := \int_0^1 \frac{(1-s)^2}{2} \sum_{k=1}^n  \big \{  \BBE  ( f^{(3)}_{n-k} (S_{k-1} + s X_k )  X_k^3 ) -  \BBE  ( f^{(3)}_{n-k} (T_{k-1} + s X_k )  X_k^3 )     \big \} ds
\]
\[
B_2 (n) := \int_0^1 \frac{(1-s)^2}{2} \sum_{k=1}^n  \big \{  \BBE  ( f^{(3)}_{n-k} (S_{k-1} + s Y_k )  Y_k^3 ) -  \BBE  ( f^{(3)}_{n-k} (T_{k-1} + s Y_k )  Y_k^3 )     \big \}  ds \, .
\]
We first handle $B_1(n)$. For any integer $k $ in $[1,n]$ and any $s \in [0,1]$, 
\begin{multline*}
 \BBE   \Big ( X_k^3 \Big \{   f^{(3)}_{n-k} (S_{k-1} + s X_k )  -   ( f^{(3)}_{n-k} (T_{k-1} + s X_k )       \Big \}   \Big )   \\ =
  \sum_{\ell =1 }^{k-1}   \BBE   \Big ( X_k^3 \Big \{   f^{(3)}_{n-k} (S_{\ell-1}   + X_\ell + T_{k-1} - T_\ell+ s X_k )  -   f^{(3)}_{n-k} (S_{\ell-1}   + Y_\ell + T_{k-1} - T_\ell+ s X_k )       \Big \}   \Big ) \\
 =  \sum_{\ell =1 }^{k-1}   \BBE   \Big ( X_k^3 \Big \{   f^{(3)}_{n-\ell -1} (S_{\ell-1}   + X_\ell + s X_k )  -   f^{(3)}_{n-\ell -1} (S_{\ell-1}   + Y_\ell + s X_k )       \Big \}   \Big )  \, .
\end{multline*}
By independence, since $\BBE ( X_\ell) = \BBE ( Y_\ell)$ and $ \BBE ( X^2_\ell) = \BBE ( Y^2_\ell)$ for any  $\ell \leq k-1$,  
\[
 \BBE   \big ( X_k^3  (X_\ell - Y_{\ell} )  f^{(4)}_{n-\ell -1} (S_{\ell-1}   + s X_k )       \big ) = 0\, , \   \BBE   \big ( X_k^3  (X^2_\ell - Y^2_{\ell} )  f^{(5)}_{n-\ell -1} (S_{\ell-1}   + s X_k )       \big ) = 0 \, .
\]
Hence, by the Taylor integral formula, 
\begin{multline*} 
 \BBE   \Big ( X_k^3 \Big \{   f^{(3)}_{n-k} (S_{k-1} + s X_k )  -   ( f^{(3)}_{n-k} (T_{k-1} + s X_k )       \Big \}   \Big )  \\
  =    \int_0^1 \frac{(1-t)^2}{2}  \sum_{\ell =1 }^{k-1}   \BBE   \Big ( X_k^3  \Big \{  X_\ell^3  f^{(6)}_{n-\ell -1} (S_{\ell-1}   +  t X_\ell + s X_k )  -  Y_\ell^3  f^{(6)}_{n-\ell -1} (S_{\ell-1}   +  t Y_\ell + s X_k )       \Big \}   \Big )  dt  \, .
\end{multline*}
Since $\BBE ( X^3_\ell) = \BBE ( Y^3_\ell)$, it follows that 
\begin{multline} \label{Dec2}
 \BBE   \Big ( X_k^3 \Big \{   f^{(3)}_{n-k} (S_{k-1} + s X_k )  -   ( f^{(3)}_{n-k} (T_{k-1} + s X_k )       \Big \}   \Big )  \\
  =    \int_0^1 \frac{(1-t)^2}{2}  \sum_{\ell =1 }^{k-1}   ( C^X_ {n,k,\ell} (s,t) -  D^X_ {n,k,\ell} (s,t) )   dt  \, ,
\end{multline}
where
\[
C^X_ {n,k,\ell} (s,t)  =   \BBE   \Big ( X_k^3  X_\ell^3   \Big \{   f^{(6)}_{n-\ell -1} (S_{\ell-1}   +  t X_\ell + s X_k )  -  f^{(6)}_{n-\ell -1} (S_{\ell-1}   + s X_k ) \Big \} \Big ) ,
\]
\[
D^X_ {n,k,\ell} (s,t)  =   \BBE   \Big ( X_k^3  Y_\ell^3   \Big \{   f^{(6)}_{n-\ell -1} (S_{\ell-1}   +  t Y_\ell + s X_k )  -  f^{(6)}_{n-\ell -1} (S_{\ell-1}   + s X_k ) \Big \} \Big )  \, .
\]
For $\ell < k$, we have
\begin{multline*}
\big  \vert C^X_ {n,k,\ell} (s,t)  \big  \vert  \leq \BBE \Big ( |  X_k^3  X_\ell^3 |  \min \big ( 2 \Vert f^{(6)}_{n-\ell -1} \Vert_{\infty}  , t |X_{\ell}| \Vert f^{(7)}_{n-\ell -1} \Vert_{\infty} \big )  \Big )  \\
 \leq \BBE ( |X|^3)  \BBE \Big (     |X|^3 \min \big ( 2 \Vert f^{(6)}_{n-\ell -1} \Vert_{\infty}  , t |X| \Vert f^{(7)}_{n-\ell -1} \Vert_{\infty} \big )  \Big ) \, .
\end{multline*}
Using the notation
\beq \label{notationunl}
u_{n, \ell} =  (n-\ell -1)/2 + v^2 
\eeq
and  Lemma \ref{smoothlemma}, we derive
\begin{multline*}
   \int_0^1 \!\!\!  \int_0^1 \frac{(1-s)^2}{2} \frac{(1-t)^2}{2}   \sum_{k =2 }^{n}   \sum_{\ell =1 }^{k-1}  \big  \vert C^X_ {n,k,\ell} (s,t)  \big  \vert  ds dt  \\
 \leq \BBE ( |X|^3)    \int_0^1 \!\!\!  \int_0^1  \frac{(1-s)^2}{2}  \frac{(1-t)^2}{2}   \sum_{k =2 }^{n}   \sum_{\ell =1 }^{k-1}   \BBE \Big (    | X|^3 \min \Big (  \frac{2  c_5 x}{ u^{5/2}_{n, \ell} }  ,
  \frac{  c_5 t |X|}{ u^{5/2}_{n, \ell} } \big )  \Big )  ds dt \\
 \leq  \frac{ c_5}{ 144} \BBE ( |X|^3)   \BBE   ( |X|^3   ( |X|    \wedge 8 x )   )   \sum_{k =0 }^{n-2}   \sum_{\ell =k }^{n-2}   \frac{2^{5/2}}{ ( \ell + 2 v^2)^{5/2}}   \, .
\end{multline*}
Next, we note that, by convexity, the following upper bound holds: for any $a \geq 1/2$,  any $m \in {\mathbb N}$ and any $p >1$, 
\[
\sum_{k \geq m} ( k + a)^{-p} \leq \int_{m+a-1/2}^{+\infty} x^{-p} dx .
\] 
Applying twice the above inequality, we obtain that
\beq \label{trivial}
\sum_{k= 0 }^{n-2}   \sum_{\ell =k }^{n-2}   \frac{1}{ ( \ell + 2 v^2)^{5/2}}  \leq \frac{2}{3}  \int_{2v^2-1}^{+ \infty} \frac{1}{x^{3/2}} dx = \frac{4}{3}  (2v^2 -1)^{-1/2}\, .
\eeq
From the above considerations, it follows that    
\begin{equation}\label{boundtotalC}
   \int_0^1 \!\!\!  \int_0^1  \frac{(1-s)^2}{2} \frac{(1-t)^2}{2}   \sum_{k =2 }^{n}   \sum_{\ell =1 }^{k-1}  \big  \vert C^X_ {n,k,\ell} (s,t)  \big  \vert  ds dt   \leq  \frac{  \sqrt{2} c_5 \BBE ( |X^3|)}{27 \sqrt{2v^2-1}}   \BBE   ( |X|^3   ( |X|    \wedge  8 x )   )     \, .
\end{equation}
On another hand, taking into account Lemma \ref{smoothlemma} and the notation \eqref{notationunl}, we derive that 
\[
\big  \vert D^X_ {n,k,\ell} (s,t)  \big  \vert  
 \leq  t  c_5 \BBE ( |X|^3)  \BBE (     Y^4)    u^{-5/2}_{n, \ell}   \, .
\]
Therefore, taking into account the upper bound \eqref{trivial}, we get 
\beq \label{boundtotalD}
   \int_0^1 \!\!\!  \int_0^1  \frac{(1-s)^2}{2} \frac{(1-t)^2}{2}   \sum_{k =2 }^{n}   \sum_{\ell =1 }^{k-1}  \big  \vert D^X_ {n,k,\ell} (s,t)  \big  \vert  ds dt  \\
 \leq    \frac{  \sqrt{2} c_5 \BBE ( |X^3|)}{27 \sqrt{2v^2-1}}   ( 3 + 2 \beta_3^2)    \, . \eeq
Starting from \eqref{Dec2} and considering  \eqref{boundtotalC} and  \eqref{boundtotalD}, we derive that 
\beq \label{conclusionDec1}
\big |  B_1(n)\big | \leq    \frac{  \sqrt{2} c_5 \BBE ( |X^3|)}{27 \sqrt{2v^2-1}}   \Big (  3 + 2 \beta_3^2 +  \BBE   ( |X|^3   ( |X|    \wedge 8 x )   )  \Big )  \, .
 \eeq
To handle the quantity  $B_2(n)$ we proceed as above. We first write
\[
B_2(n) = \sum_{k=1}^n  \sum_{\ell =1 }^{k-1} \int_0^1\!\! \int_0^1 \frac{(1-s)^2}{2}  \frac{(1-t)^2}{2}   ( C^Y_ {n,k,\ell} (s,t) -  D^Y_ {n,k,\ell} (s,t) )   ds dt
  \, ,
\]
where
\[
C^Y_ {n,k,\ell} (s,t)  =   \BBE   \Big ( Y_k^3  X_\ell^3   \Big \{   f^{(6)}_{n-\ell -1} (S_{\ell-1}   +  t X_\ell + s Y_k )  -  f^{(6)}_{n-\ell -1} (S_{\ell-1}   + s Y_k ) \Big \} \Big ) ,
\]
\[
D^Y_ {n,k,\ell} (s,t)  =   \BBE   \Big ( Y_k^3  Y_\ell^3   \Big \{   f^{(6)}_{n-\ell -1} (S_{\ell-1}   +  t Y_\ell + s Y_k )  -  f^{(6)}_{n-\ell -1} (S_{\ell-1}   + s Y_k ) \Big \} \Big )  \, .
\]
Taking into account the previous computations, we infer that 
\[
\big  \vert C^Y_ {n,k,\ell} (s,t)  \big  \vert  \leq  \BBE ( |Y|^3)  \BBE \Big (     |X|^3 \min \big ( 2 \Vert f^{(6)}_{n-\ell -1} \Vert_{\infty}  , t |X| \Vert f^{(7)}_{n-\ell -1} \Vert_{\infty} \big )  \Big ) \, .
\]
Therefore
\[
   \int_0^1 \!\!\!  \int_0^1  \frac{(1-s)^2}{2} \frac{(1-t)^2}{2}   \sum_{k =2 }^{n}   \sum_{\ell =1 }^{k-1}  \big  \vert C^Y_ {n,k,\ell} (s,t)  \big  \vert  ds dt    \leq  \frac{\sqrt{2} c_5 \BBE ( |Y|^3)}{27 \sqrt{2v^2-1}}    \BBE   ( |X|^3   ( |X|    \wedge  8x )   )     \, .
\]
On another hand
\[
\big  \vert D^Y_ {n,k,\ell} (s,t)  \big  \vert  \leq    t  c_5 \BBE ( |Y|^3)  \BBE (     Y^4)    u^{-5/2}_{n, \ell} \, ,
\]
implying that  
\[
 \int_0^1 \!\!\!  \int_0^1  \frac{(1-s)^2}{2} \frac{(1-t)^2}{2}   \sum_{k =2 }^{n}   \sum_{\ell =1 }^{k-1}  \big  \vert D^Y_ {n,k,\ell} (s,t)  \big  \vert  ds dt  \\  \leq    \frac{\sqrt{2} c_5 \BBE ( |Y|^3)}{27 \sqrt{2v^2-1}}   ( 3 + 2 \beta_3^2)   \, .
\]
So, overall, 
\beq \label{conclusionDec1bis}
\big |  B_2(n)\big | \leq    \frac{\sqrt{2} c_5 \BBE ( |Y|^3)}{27 \sqrt{2v^2-1}}    \Big (  3 + 2 \beta_3^2 +   \BBE   ( |X|^3   ( |X|    \wedge  8 x )   )  \Big )  \, .
 \eeq
Starting from  \eqref{decwithAn} and considering the upper bounds \eqref{BoundAn}, \eqref{conclusionDec1} and \eqref{conclusionDec1bis}, it follows that 
\beq  \label{boundofzeta}
 | \Delta_{n,v} (f) |  \leq   \Big (  \frac{\sqrt{2} c_5 \BBE ( |X|^3 + |Y|^3)}{27 \sqrt{2v^2-1}}    +   \frac{c_2}{12}   \Big )   \Big (  3 + 2 \beta_3^2   +  \BBE   ( |X|^3   ( |X|    \wedge 8 x )   )  \Big )    \, .
\eeq
In order to minimize the numerical constants, we now set 
\beq \label{choicev2}
v^2 = 0.50 + 0.15 \bigl( \BBE ( |X|^3 ) + \BBE ( |Y|^3 ) \bigr)^2 . 
\eeq
This choice, combined with \eqref{ine1}, \eqref{ineforRn}, \eqref{boundofzeta} leads to  the following proposition: 
\begin{Proposition} \label{boundfinal0}
Let 
\[
B(n,x) : =  0.50 + 0.15 \bigl( \BBE ( |X|^3 ) + \BBE ( |Y|^3 ) \bigr)^2 
+  \alpha_1  \bigl (  3 + 2 \beta_3^2   +   \BBE   ( |X|^3   ( |X|    \wedge 8 x )   )  \bigr )    \, ,
\]
with $\alpha_1 =   \sqrt{20/3} (c_5/27)  +   (c_2/12)   \leq \gamma_2 := 0.64584$. Then 
\[
  \sup_{f \in  {\mathcal F}_x}\big | \BBE ( f(S_n ) ) -  f(T_n))  \big |  \leq   B(n,x) \, . 
\]
\end{Proposition}
In order to complete the proof of Theorem \ref{Thmkappa}, we now give an upper bound on $B (n,x)$ depending more explicitely on the law of $X$.
First 
$$
\bigl( \BBE ( |X|^3 ) + \BBE ( |Y|^3 ) \bigr)^2 \leq 4 (\BBE ( |X|^3 ) )^2 + (4/3) (\BBE ( |Y|^3 ) )^2 ,
$$
and second, since $ \BBE( Y^2) =1$, 
$(\BBE ( |Y|^3 ) )^2 \leq \BBE( Y^4) = 3 + 2 \beta_3^2$.
Hence 
\beq \label{UpperboundB(n,x)}
B (n,x) \leq 1.1 + 0.6 ( \BBE ( |X|^3 ))^2 + 0.4 \beta_3^2 +  \alpha_1  \bigl (  3 + 2 \beta_3^2   +   \BBE   ( |X|^3   ( |X|    \wedge 8 x )   )  \bigr )    \, .
\eeq

When $\sigma^2 =1$ and $\beta_3 \neq 0$, Theorem \ref{Thmkappa} follows by considering \eqref{ineconvexity} together with Proposition \ref{boundfinal0}, \eqref{UpperboundB(n,x)} and inequality \eqref{boundfinal2}. 

Assume now that $\beta_3=0$. In this case, first we add an independent Gaussian r.v. $Z$ centered and with variance $v^2$, and then we take  $(Y_k)_{k \geq 1}$ iid standard Gaussian random variables that are independent of $(X_k)_{k \geq 1}$.   We proceed as for the proof of the case $\beta_3 \neq 0$, but with $u_n = n-1 + v^2$ and $u_{n, \ell} =  n-\ell -1 + v^2 $.   So, overall, we infer that when $\sigma^2=1$ and $\beta_3=0$, 
\[
\kappa_{\psi_{x}} ( P_{S_n} ,  P_{G_{n} } )  \leq   {\tilde \gamma}_1  +  0.6 (\BBE (|X^3|))^2      + {\tilde \gamma}_2  \BBE   ( |X|^3   ( |X|    \wedge 8 x )   )  \, ,
\]
where ${\tilde \gamma}_1  \leq \gamma_1$ and  ${\tilde \gamma}_2  \leq \gamma_2$. This ends the proof of the theorem. \qed

\subsection{Proof of Remark \ref{remark2}} \label{demrem22}
Since $\varphi'$ is concave and then absolutely continuous, there exists a  positive measure $\nu$ on ${\mathbb R}_+$ such that 
$\varphi''(0)  -  \varphi''(z) =  \nu (]0,z[)$ almost surely.
Since $\varphi''(0)=1 $ and $\varphi'(0)=0$, 
$x - \varphi'(x)  =  \int_0^x \nu(]0,t[) dt$.
It follows that 
\[
1 - \lim_{x \rightarrow \infty}  x^{-1} \varphi'(x)=   \lim_{x \rightarrow \infty}  x^{-1} \int_0^x \nu(]0,t[) dt \, .
\]
Since  $\lim_{x \rightarrow \infty} x^{-1} \varphi'(x) =0$, it follows that $\nu$ is a probability measure. 
Hence
\beq \label{varphiwithpsiprime}
\varphi'(x)  =  \int_0^x \nu ([y, + \infty [ ) dy =  \int_0^{\infty} (x \wedge y) d \nu (y) =   \frac{1}{2} \int_0^{\infty}  \psi'_y (2x) d \nu (y) \, ,
\eeq
where $\psi_y$ is defined in \eqref{defpsi}.  
Since $\varphi' (0) = 0$, it implies that 
\[ \varphi(z) = \frac{1}{2} \int_{ {\mathbb R}_+} \psi_t( 2 z ) d \nu (t) = 2 \int_{ {\mathbb R}_+} \psi_{t/2}(  z ) d \nu (t)  \, , 
\] 
where the second equality comes from \eqref{inepsi}. \qed

\subsection{Proof of Theorem \ref{coravecphi}}  By Lemma \ref{Wphidistance},  
\beq \label{Th2-1}
W_{\varphi}( P_{S_n} , P_{G_{n \sigma^2 } }  ) \leq  W_{\varphi}( P_{S_n} , P_{\sigma T_n} )  +  W_{\varphi}( P_{ \sigma T_n} , P_{G_{n\sigma^2} } )  \, , 
\eeq
where $T_n$ is defined as in the proof of Theorem \ref{Thmkappa} but with $\beta_3 = \BBE (X^3) / \sigma^3 $ in \eqref{defmomentsY}.   

Since $\varphi$ is in $\Psi$, $\varphi(x) \leq x^2 /2$. Whence, taking into account \eqref{boundfinal1}, 
\beq \label{Th2-2}
W^2_{\varphi}(  P_{ \sigma T_n} , P_{G_{n\sigma^2} } ) 
 \leq \frac{\sigma^2}{2} W^2_{2}(  P_{  T_n} , P_{G_n } )  \leq  1.874  \,   \frac{\mu_3^2}{\sigma^4} \leq \bigl( 1.369 \sigma^{-2} \mu_3 \bigr)^2 \, .
\eeq
We now handle the quantity $W_{\varphi}( P_{S_n} , P_{\sigma T_n} ) $.  Taking into account Remark \ref{remark2}, 
\[
\kappa_{\varphi}( P_{ S_n} , P_{\sigma T_n} )  = 2  \int_{ {\mathbb R}_+}   \kappa_{  \psi_{t/2}  } \big (  P_{ S_n} , P_{  \sigma T_n }  \big ) d \nu (t) \, .
\]
Applying inequality \eqref{toassumesigma1} with $a = \sigma$, we get that
\[
\kappa_{  \psi_{t/2}  }( P_{S_n} , P_{\sigma T_n} ) = \sigma^2  \kappa_{ \psi_{t/(2 \sigma)}} ( P_{ \sigma^{-1}S_n} , P_{ T_n} ) \, .
\]
Hence
\[
\kappa_{\varphi}( P_{ \sigma^{-1}S_n} , P_{\sigma T_n} )  = 2  \sigma^2 \int_{ {\mathbb R}_+}   \kappa_{  \psi_{t/(2 \sigma)}  } \big (  P_{  \sigma^{-1}S_n} , P_{  T_n }  \big ) d \nu (t) \, .
\]
Now, using Proposition \ref{boundfinal0} and \eqref{UpperboundB(n,x)}, 
\beq \label{Th2-3}
\kappa_{\varphi}( P_{S_n} , P_{\sigma T_n} ) \leq  2 C_0 (X) +   8 \gamma_2 \sigma^{-2} \int_{ {\mathbb R}_+}  \BBE  \big ( |X|^3 \min (  |X|/4 ,  t  ) \big )  d \nu (t)  \, . 
\eeq
Recall that   $ \varphi'(|z|) =\int_{ {\mathbb R}_+}  ( |z| \wedge t) d \nu (t)$ (see \eqref{varphiwithpsiprime}). Hence
\beq \label{Th2-4}
\int_{ {\mathbb R}_+}  \BBE  \big ( |X|^3 \min ( |X| /4 , t ) \big )  d \nu (t) \leq   \BBE  \big ( |X|^3 \varphi' (  |X| /4  ) \big ) \, .
\eeq
Starting from \eqref{Th2-1}, the theorem follows by taking into account \eqref{Th2-2}, \eqref{Th2-3} and \eqref{Th2-4}. 
 \qed

\subsection{Proof of Remark \ref{remTh22}} \label{section3.5}
As quoted in the proof  of Lemma \ref{Wphidistance}, 
$x \mapsto x^{-2} \varphi (x) $ is nonincreasing on $]0, + \infty [$. Hence the derivative of $x \mapsto x^{-2} \varphi (x) $ is nonpositive which ensures that 
$\varphi'(x) \leq   2 x^{-1} \varphi (x)  $.  Therefrom
\[
\BBE  \big ( |X|^3 \varphi' (  |X| /4  ) \big )\leq  8  \BBE  \big ( |X|^2 \varphi (  |X|  /4 ) \big ) \, . \quad \quad  \qed
\]

\subsection{Proof of Proposition \ref{coravecxp}}
We first notice that 
\beq \label{Wpdistance}
W_p ( P_{S_n} , P_{G_{n\sigma^2}} ) = \sigma W_p ( P_{S_n/\sigma} , P_{G_{n}} ) \leq  \sigma W_p ( P_{S_n/\sigma} , P_{T_n} ) + \sigma  W_p  ( P_{T_n} , P_{G_n} ) \, , 
\eeq
where $T_n$ has been defined in the proof of Theorem \ref{Thmkappa}. Using the fact that $p \in ]1,2[$ and the upper bound \eqref{boundfinal1}, we get 
\beq \label{WpTnGn}
\sigma W_p  ( P_{T_n} , P_{G_{n}} ) \leq \sigma  W_2  ( P_{T_n} , P_{G_{n}} ) \leq   1.936  \sigma^{-2} | \mu_3 |  .
\eeq
In order to bound up the first term on right hand in \eqref{Wpdistance}, we define the rescaled functions $g_{p,a}$ from the function $g_p$
defined in \eqref{defgp}  by 
$g_{p,a} (x) = a^2 g_p (x/a)$ for any $a>0$. Note that the functions $g_{p,a}$ are in the class $\Psi$. In addition, using inequality \eqref{comparisongpxp}, we get that 
\beq
\label{compxpgpa}
x^p \leq a^{p-2} p g_{p,a} (x) + a^p ( 1 -p/2 ).
\eeq
Taking into account the upper bounds \eqref{Th2-3} and \eqref{Th2-4}, it follows that 
\beq 
\kappa_p ( P_{S_n/\sigma} , P_{T_n} )  \leq ( 1-p/2) a^p + 2p a^{p-2} C_0 \Bigl( \frac {X}{ \sigma } \Bigr) + 8 \gamma_2 p a^{p-2} 
\BBE \Bigl( \Big| \frac{X}{\sigma} \Big|^3 g'_{p,a} \Bigl( \frac{ |X| }{4\sigma} \Bigr) \Bigr)  .
\eeq
Next,  since $g'_p (z) = \min ( z, z^{p-1}) \leq z^{p-1}$ for any $z\geq 0$, we have that 
\[
g'_{p,a} (z) = a g'_p (z/a) \leq a^{2-p} z^{p-1}
\]
 for any $z\geq 0$.
Hence 
\beq \label{Prop2.2-1}
\kappa_p ( P_{S_n/\sigma} , P_{T_n} )  \leq ( 1-p/2) a^p + 2p a^{p-2} C_0 (X/\sigma) + 8 p 4^{1-p} \gamma_2
\BBE \bigl( |X/\sigma|^{2+p} \bigr) .
\eeq
In order to minimize the above upper bound, we then choose $a = \sqrt{4 C_0 (X/\sigma)}$, which gives 
\beq \label{Prop2.2-2}
\kappa_p ( P_{S_n/\sigma} , P_{T_n} ) \leq  2^p (C_0 (X/\sigma) )^{p/2} + p 2^{5-2p} \gamma_2 \BBE \bigl( |X/\sigma|^{2+p} \bigr) .
\eeq
Starting from \eqref{Wpdistance} and considering the upper bounds \eqref{WpTnGn} and  \eqref{Prop2.2-2},  the proposition follows  by noticing that $\sigma^2 C_0 (X/\sigma) = C_0 (X)$.  \qed

\subsection{Proof of Corollaries \ref{asymptoticconstants} and \ref{asymptotickappa1g}}

In this subsection, we prove Corollary \ref{asymptoticconstants} and item (a) of Corollary \ref{asymptotickappa1g}. 
The proof of item (b) of Corollary \ref{asymptotickappa1g}, being similar, will be omitted. 
We start by proving the following lemma.

\begin{Lemma} \label{lmaUI} Let
$(Z_n)_{n \geq 1}$ be defined in \eqref{defZn}. Under the conditions of Corollary \ref{asymptoticconstants}, $ (\varphi (Z_n) )_{n \geq 1}$ is uniformly integrable. 
\end{Lemma}
\noindent \textbf{Proof of Lemma \ref{lmaUI}.} We have to prove that 
\beq \label{butUI}
\lim_{M \rightarrow \infty} \sup_{n \geq 1} \BBE  (\varphi (Z_n) {\bf 1}_{  Z_n \geq M } )  =0 \, .
\eeq
Let $M >0$ and define the function $\psi_M$ on ${\mathbb R}$ by the following conditions: $\psi_M$ is even, $\psi_M (0)  =0$ and its derivative on ${\mathbb R}_+$ is defined by  
\[
\psi'_M (x) =  \left\{
\begin{aligned}
& x & \text{ if $ x \leq M$,}  \\  
&  \frac{M}{ \varphi' (M)}  \varphi' (x) & \text{ if $x \geq M$.} \\
  \end{aligned}
\right.
\]
Clearly  $\psi_M$ belongs to the class $\Psi$. Now applying Theorem \ref{coravecphi} to the function $\psi_M$, it follows that there exists a universal 
constant  $C$ such that for any $n \geq 1$,
\beq \label{UIp1}
\BBE  (\psi_M (Z_n)  )  \leq   C \Big (  \BBE ( |X|^3)     +   \BBE  \big ( |X|^3 \psi_M' (  |X|   ) \big )  \Big ) \, .
\eeq
On another hand, since $\varphi$ is convex and $\varphi(0)=0$, $\varphi(x) \leq x \varphi' (x)$ for $ x \geq 0$. This implies that  
\[
\BBE  (\varphi (Z_n) {\bf 1}_{  Z_n \geq M } ) \leq  \BBE  ( Z_n \varphi' (Z_n) {\bf 1}_{  Z_n \geq M } ) \leq  \frac{ \varphi' (M)}{M}  \BBE  (  Z_n \psi_M' (Z_n)  ) \, .
\]
Now, as quoted in Section \ref{section3.5}, since  $\psi_M$ belongs to the class $\Psi$, $ x \psi_M' (x) \leq 2  \psi_M (x) $. Hence
\beq \label{UIp2}
\BBE  (\varphi (Z_n) {\bf 1}_{  Z_n \geq M } ) \leq    2  \frac{ \varphi' (M)}{M}  \BBE  (   \psi_M (Z_n)  )   \, .
\eeq
Combining \eqref{UIp1}  and \eqref{UIp2}, it follows that 
\beq \label{UIp3}
\BBE  (\varphi (Z_n) {\bf 1}_{  Z_n \geq M } ) \leq    2 C \Big (   \frac{ \varphi' (M)}{M}    \BBE ( |X|^3)   +  \frac{ \varphi' (M)}{M}    \BBE  \big ( |X|^3 \psi_M' (  |X|   ) \big )   \Big )   \, .
\eeq
Note that 
\[
 \frac{ \varphi' (M)}{M}    |X|^3 \psi_M' (  |X|   )  =  \frac{ \varphi' (M)}{M}    |X|^4  {\bf 1}_{  | X| \leq M }  + 
  |X|^3 \varphi' ( |X| )   {\bf 1}_{  | X|  > M }    \, .
\]
Now, since $ x \mapsto x^{-1} \varphi' (x)   $ is nonincreasing, 
\[
 \frac{ \varphi' (M)}{M}    |X|^4  {\bf 1}_{  | X| \leq M } +    |X|^3 \varphi' ( |X| )   {\bf 1}_{  | X|  > M }    \leq     |X|^3 \varphi' ( |X| )   \, .
\]
Next, since $  \lim_{M \rightarrow + \infty} M^{-1}  \varphi' (M) =0$, the term of the left hand side of the above inequality tends to $0$ as $M \rightarrow + \infty$. Hence, since 
$\BBE ( |X|^3 \varphi' ( |X| )  \big ) < \infty$, by the dominated convergence theorem, we get that 
\beq \label{UIp4}
 \lim_{M \rightarrow + \infty}    \frac{ \varphi' (M)}{M}    \BBE  \big ( |X|^3 \psi_M' (  |X|   ) \big )  = 0 \, .
\eeq
Starting from \eqref{UIp3} and considering \eqref{UIp4}, the convergence \eqref{butUI} follows. \qed

\medskip\no
{\bf Proof of Item (a) of Corollary \ref{asymptoticconstants}}. Let $G = n^{-1/2}G_{n}$ and $U = \Phi ( G)$ where $\Phi$ is the c.d.f. of a standard normal.  
Then $U$ has the uniform distribution over $[0,1]$ and 
\[
Z_n =  | F_n^{-1} (U) - \sqrt{n} \Phi^{-1} (U) | \, .
\]
Set
\beq \label{defofZ}
Z  = \big |   \frac{\mu_3}{6}  (  ( \Phi^{-1} (U) )^2 -1 ) \big | \, .
\eeq
According to the Cornish-Fisher expansion (see for instance Lemma 2.1 (a) in Rio \cite{Rio11}  or  \cite[Theorem 2.4]{Hall92} and final comments to Chapter 2), we get that 
\[
Z_n \rightarrow^{\mathcal L} Z , \text{ as } n \rightarrow \infty \, ,
\]
 where $\rightarrow^{\mathcal L}$ means convergence in distribution. Since $\varphi $ is a continuous function, by the continuous mapping theorem, we also have that $\varphi(Z_n) \rightarrow^{\mathcal L} \varphi(Z) $, as $ n \rightarrow \infty $. 
Together with  Lemma \ref{lmaUI}  and the convergence of moments Theorem (see Theorem 3.5 in \cite{Bi99}), this convergence in distribution implies that 
$\BBE ( \varphi(Z_n) ) \rightarrow \BBE ( \varphi(Z) ) $, as $ n \rightarrow \infty $. This ends the proof of Item (a) of Corollary \ref{asymptoticconstants}. 



\medskip\no
{\bf Proof of Item (b) of Corollary \ref{asymptoticconstants}}. Let
\[
 {\bar Z}_n =  |  {\bar F}^{-1}_n  (U)   -  \sqrt{n} \Phi^{-1} (U) | 
\]
where ${\bar F}_n$ is the cdf of ${\bar S}_n= S_n +hV$.   Recalling the notation \eqref{defofZ}, we have ${\bar Z}_n \rightarrow^{\mathcal L} Z $, as  $n \rightarrow \infty$ (see for instance Lemma 2.1 (b) in Rio \cite{Rio11}). 
Whence, following the proof of  Item (a), it is enough to prove the uniform integrability of  
$(\varphi ( {\bar Z}_n)  )_{n \geq 1}$.  With this aim, we first notice that  for any $n \geq 1$, 
\[
 {\bar Z}_n  \leq   Z_n  + | {\bar F}_n^{-1} (U)  -   { F}_n^{-1} (U)  |  \, .
\]
 Next,  
since $V$ is a random variable with uniform distribution over  $[-1/2 , 1/2]$ and $\bar S_n = S_n + hV$, 
\[
S_n - h/2 \leq \bar S_n \leq S_n + h/2 \  \text{ almost surely, }
\]
which implies that 
$F_n^{-1} (u) - (h/2) \leq \bar F_n^{-1} (u) \leq F_n^{-1} (u) + (h/2)$
for any $u$ in $]0,1[$. Hence 
$| {\bar F}_n^{-1} (U)  -   { F}_n^{-1} (U) | \leq h/2$  almost surely,
which ensures that 
\beq \label{diffZbarZ}
\bar Z_n \leq Z_n +  h/2 \ \text{ almost surely.}
\eeq
Now, using both the above inequality and \eqref{subadvarphi}, 
$\varphi ( \bar Z_n) \leq 2 \varphi ( Z_n) + 2 \varphi ( h/2)$ almost surely.
Since $(\varphi (Z_n))_{n\geq 1}$ is uniformy integrable, this proves the  uniform integrability of 
$(\varphi (\bar Z_n))_{n\geq 1}$ and ends the proof of Item (b) of Corollary \ref{asymptoticconstants}.   

\medskip\no
{\bf Proof of Item (a) of Corollary \ref{asymptotickappa1g}}. By Lemma 2.1(a) in Rio \cite{Rio11}, for any positive $\varepsilon$, 
the sequence of  functions $(F_n^{-1}- \sqrt{n} \Phi^{-1})_{n\geq 1}$ converges uniformly to $\mu_3 ( (\Phi^{-1})^2 - 1) /6$
 on $[\varepsilon, 1- \varepsilon]$. This implies that the above sequence of functions is uniformly bounded over $[\varepsilon , 1- \varepsilon]$. 
Next $g (U)$ belongs to $L_{\varphi}^*$ and $L_{\varphi^*} \subset L^1$, which ensures the $g$ is integrable over $[0,1]$. Hence, 
by the dominated convergence theorem, 
\beq \label{Lebesguecv}
\lim_{n\rightarrow \infty} \int_\varepsilon^{1-\varepsilon} g(u) |F_n^{-1} (u) - \sqrt{n} \Phi^{-1} (u) | du = \frac{ |\mu_3|}{6} 
\int_\varepsilon^{1-\varepsilon} g(u) | ( \Phi^{-1} (u) )^2 - 1 | du . 
\eeq
\par
It remains to prove that 
\beq \label{tensionkappa1g}
\lim_{\varepsilon \rightarrow 0} \sup_{n\geq 1} \int_{]0,1[ \setminus [\varepsilon, 1- \varepsilon]} g(u) |F_n^{-1} (u) - \sqrt{n} \Phi^{-1} (u) | du = 0.
\eeq
Let $a$ be a positive real such that $\BBE \bigl( \varphi^* (a g(U) ) \bigr)< \infty$. By the Young inequality $xy \leq \varphi^* (x) + \varphi (y)$
applied to $x = ag(U)$ and $y =  |F_n^{-1} (u) - \sqrt{n} \Phi^{-1} (u) |$, 
\[
a  \int_{]0,1[ \setminus [\varepsilon, 1- \varepsilon]} g(u) |F_n^{-1} (u) - \sqrt{n} \Phi^{-1} (u) | du \leq 
I ( \varepsilon)  + J_n ( \varepsilon) ,
\]
with 
\[ I (\varepsilon) =  \int_{]0,1[ \setminus [\varepsilon, 1- \varepsilon]} \varphi^* (ag(u) ) du \quad \text{and} \quad 
J_n ( \varepsilon) = \int_{]0,1[ \setminus [\varepsilon, 1- \varepsilon]} \varphi ( F_n^{-1} (u) - \sqrt{n} \Phi^{-1} (u) ) du . 
\]
Now, Lemma \ref{lmaUI} ensures that $\lim_{\varepsilon \rightarrow 0} \sup_{n\geq 1} J_n (\varepsilon) = 0$. Next,
since $u \mapsto \varphi^* (ag(u))$ is integrable over $]0,1[$,
$\lim_{\varepsilon \rightarrow 0} I (\varepsilon) = 0$, 
which completes the proof of \eqref{tensionkappa1g}. Now \eqref{Lebesguecv} and  \eqref{tensionkappa1g} imply 
Item (a) of Corollary \ref{asymptotickappa1g}.  \qed

\subsection{Proof of Theorem \ref{UpperBoundHLtail}}

We start with the following lemma, which is of independent interest. The proof of this lemma is postponed to Appendix.  

\begin{Lemma} \label{lemmaHtilde} Let $A$ and $B$ be two real-valued random variables. Then for any real numbers $x$ and $t$,
$\tilde H_{A+B} (x)  \leq \max \big (  \tilde H_{A} (t) , \tilde H_{B}  (x-t)  \big )$.
\end{Lemma}

\begin{Remark} The classical inequality 
${\mathbb P} ( A+B \geq x) \leq {\mathbb P} ( A \geq t) +  {\mathbb P} ( B > x-t)$
for usual tail functions
cannot be improved, according to Proposition 1 in Pinelis \cite{Pi2019}. The above lemma proves that for sums of r.v.'s. $\tilde H$ has a better behavior than the tail function.  
\end{Remark}

We go back to the proof of Theorem \ref{UpperBoundHLtail}. We first note that 
\[
\tilde H_{Z_n} (t)  = \tilde H_{Z_n/\sigma} (t/ \sigma)   \, .
\]
We shall apply Lemma \ref{lemmaHtilde}  with $A = | \sigma^{-1}S_n - T_n | $ and 
$B= | T_n - G_{n } | $ where $T_n$ be the random variable defined in the proof of Theorem \ref{Thmkappa}. For any $\alpha \in ]0,1[$ and any positive $t$, we then get
\beq \label{th2.3ine0}
\tilde H_{Z_n} (t)   \leq  \max \big (  \tilde H_{X} ( | \sigma^{-1}S_n - T_n|) ( \alpha t/ \sigma) , \tilde H_{|T_n - G_{n }|}  ((1-\alpha)  t/ \sigma)  \big )  \, .
\eeq
Next,  for any $s >0$, note that $(x-s)_+ \leq  s^{-1} \psi_s (x)$ where $\psi_s$ is defined in \eqref{defpsi}. Hence, for any  real-valued random variable $Z$ and any positive real $u$,
\beq \label{ineHZu}
\tilde H_{|Z|} (u)   \leq \inf_{s < u } \frac{1}{(u -s) s } \BBE ( \psi_s (|Z|) )   \leq  \frac{4}{ u^2 } \BBE ( \psi_{u/2} (|Z|) )   \, .
\eeq
On one hand, taking into account \eqref{ineHZu} and \eqref{boundfinal1}, we get
\beq \label{th2.3ine1}
\tilde H_{|T_n - G_{n }|}  \Big  ( \frac{(1-\alpha)  t}{ \sigma } \Big )    \leq   \frac{4 \sigma^2}{(1-\alpha)^2  t^2 }  \kappa_{\psi_{u/2} }   (P_{T_n}, P_{G_n}) \leq   \frac{4 \times 0.937}{(1-\alpha)^2  t^2 }  \frac{\lambda_3^2}{ \sigma^4} \, .
\eeq
On another hand, taking into account  \eqref{ineHZu} together with Proposition  \eqref{boundfinal0}, the upper bound \eqref{UpperboundB(n,x)} and the fact that 
$ |\BBE ( X^3) | \leq \lambda_3$ and $\sigma^3 \leq  \lambda_3$, we infer that 
\beq \label{th2.3ine2}
\tilde H_{ | \sigma^{-1}S_n - T_n| }  \Big  ( \frac{\alpha t }{ \sigma } \Big )     \leq   \frac{4 }{ \alpha^2  t^2 } \Big ( ( \gamma_1 + \frac{3}{5} + \gamma_3)   \frac{\lambda_3^2}{ \sigma^4}  +  \frac{\gamma_2}{ \sigma^2}   \BBE   ( |X|^3   ( |X|    \wedge 4 \alpha t )   ) \Big ) \, .
\eeq
In the  inequality above the constants $\gamma_1$, $\gamma_2$ and $\gamma_3$ are those involved in the statement of Theorem \ref{Thmkappa}. In particular the constant  
$  a := \gamma_1 + 3/5 + \gamma_3   $ can be chosen equal to $5.3294$.  We now choose $\alpha = \tilde \alpha$ such that 
\beq \label{conditionalpha}
\frac{a }{  \tilde \alpha^2  t^2 } \geq  \frac{ 0.937}{(1- \tilde \alpha)^2  t^2 }  \, .
\eeq
Numerical computation  gives that $\tilde \alpha = 0.7045699$ satisfies \eqref{conditionalpha}. 
Then starting from \eqref{th2.3ine0} and taking into account \eqref{th2.3ine1} and \eqref{th2.3ine2}, we derive that, for any positive real $t$, 
\[
\tilde H_{Z_n} (t)  \leq  \frac{4 a }{   \tilde \alpha^2  t^2 }  \frac{\lambda_3^2}{ \sigma^4}  +  \frac{ 4 \gamma_2}{ \tilde \alpha^2  \sigma^2 }   \BBE   ( |X|^3   ( |X|    \wedge 4  \tilde \alpha t )   ) \, . 
\]
This ends the proof  by taking into account the values of $a$, $ \tilde \alpha$ and $\gamma_2$. \qed

\subsection{Proof of Corollary \ref{corweakmoments}}
According to Theorem \ref{UpperBoundHLtail} and the fact that $\tilde H_{Z_n} (t)  \leq 1$, we have 
\[
t^p \tilde H_{Z_n} (t) \leq t^p \min ( 1 , 42.943 \, \sigma^{-4} \lambda_3^2 t^{-2} )  + 5.2041 t^{p-2}  \, \sigma^{-2} \BBE \bigl( |X|^3 \min ( |X| , 2.8183 t ) \bigr) 
\]
since  $\min (a+b, 1) \leq  \min (a, 1) + b$ for any nonnegative reals $a$ and $b$.   Next,  for $p \in ]1,2[$,   note that 
\beq \label{comparisonweakmoment}
 \BBE  \big ( |X|^3 \min ( |X| , u ) \big )  \leq  (p+2) \bigl ( (2-p)(p-1) \bigr)^{-1}  \Lambda_{p+2} (X) u^{2-p} \, .
\eeq 
Whence
\[
t^p \tilde H_{Z_n} (t) \leq t^p \min ( 1 , 42.943 \, \sigma^{-4} \lambda_3^2 t^{-2} )  +  a_2 (p+2) \bigl ( (2-p)(p-1) \bigr)^{-1}  \sigma^{-2} \Lambda_{p+2} (X) .
\]
The supremum over $t >0$ is reached at $t_0 $ such that $1 =  42.943 \, \sigma^{-4} \lambda_3^2 t_0^{-2}$. Numerical computation shows that
$t_0 \leq 6.5531   \, \sigma^{-2} \lambda_3$. This proves the first part of the corollary. When the random variables have strong moment of order $p+2$,  we use the fact that 
\[\BBE \bigl( |X|^3 \min ( |X| , 2.8183 t ) \bigr) \leq   (2.8183 t )^{2-p} \BBE (|X|^{p+2}) , \]
and the same arguments as before. 
\qed

\subsection{Proof of Corollary \ref{asymptoticconstantsforLambda}}

The proof is based on the following lemma.

\begin{Lemma} \label{lmaUConQ}  Under the conditions of Corollary \ref{asymptoticconstantsforLambda} and if the distribution of $X$ is not a lattice distribution, then: 
\begin{itemize}
\item[(a)] for any $\varepsilon >0$, the sequence of functions $ ( Q_{Z_n} )_{n \geq 1}$ converges uniformly to $Q_{|G^2 -1|}  \times \frac{| \BBE ( X^3)|}{6} $ over $[\varepsilon, 1]$.
\item[(b)]  $\lim_{u \rightarrow 0} \sup_{n \geq 1} u^{1/p} {\tilde Q}_{Z_n}  (u) = 0$. 
\end{itemize}
\end{Lemma}
\noindent \textbf{Proof of Lemma \ref{lmaUConQ}.} We start with the proof of Item (a).  Recall that by the Cornish-Fisher expansion, $ 
Z_n \rightarrow^{\mathcal L} Z $ as $ n \rightarrow \infty $, where $Z=|G^2 -1| \times \frac{| \BBE ( X^3)|}{6} $. Since $Z$ is a continuous r.v., it follows that, for any $u \in ]0,1] $, 
$Q_{Z_n} (u) \rightarrow Q_{Z} (u) =  Q_{|G^2 -1|} (u)  \times \frac{| \BBE ( X^3)|}{6} $, as $n \rightarrow \infty$ (see for instance Lemma 21.2 in \cite{VdV98}).  
Note now that for any positive integer $n$, $u \mapsto Q_{Z_n} (u) $ is nonincreasing. Whence, by the second Dini's Theorem, the convergence is also uniform over all intervals of the form $[\varepsilon,  1]$ for any $\varepsilon >0$ and Item (a) is proved.

We turn now to Item (b). We first prove that 
\beq \label{inequalitywithtildeH}
\lim_{t \rightarrow \infty} t^p  \sup_{n \geq 1} {\tilde H}_{Z_n} (t) = 0 \, .
\eeq
By Theorem \ref{UpperBoundHLtail}, there exists a positive constant $C$ such that for any positive $t$,
 \[
t^2 \sup_{n \geq 1} \tilde H_{Z_n} (t) \leq C \Big (  1 +  \BBE \bigl( |X|^3 \min ( |X| , t ) \bigr)  \Big ).
\]
Next 
\[
\lim_{t \rightarrow \infty} t^{p-2} \BBE \bigl( |X|^3 \min ( |X| , t ) \bigr)  = 0 \, .
\]
For $p \in ]1,2[$, this follows from the fact that  $H_{|X|} (x) = o ( x^{-p-2} )$, and for $p=1$ this follows from the dominated convergence theorem and the fact that 
$\BBE (|X|^3) < \infty$. Now the above convergence  implies  \eqref{inequalitywithtildeH}. Next, for all $\eta >0$ there exists $A >0$ such that for any $n \geq 1$ and any $x \geq A$, 
$ \tilde H_{Z_n} (x) < \eta^p  x^{-p}$.  Hence, according to inequality \eqref{GeneralizedInverse},  $ \tilde Q_{Z_n} (\eta^p  x^{-p}) < x$ for any $x \geq A$. This implies that 
$ \eta  x^{-1} \tilde Q_{Z_n} (\eta^p  x^{-p}) < \eta$ for any $x \geq A$.  It follows that for any $u < \eta^p A^{-p}$,  $u^{1/p} \tilde Q_{Z_n} (u) < \eta$, for any $n \geq 1$. This ends 
the proof of Item (b) of Lemma \ref{lmaUConQ}. \qed

\medskip

We now return to the  proof of  Corollary \ref{asymptoticconstantsforLambda}.  Let us start by Item (a). According to \eqref{weakmomentsq} and the fact that  $ H_Z (x) < u$ iff $x> \tilde Q_Z (u) $, we have $\Lambda_p (Z_n) = \sup_{u \in ]0,1]} u  { Q}^p_{Z_n}  (u) $.  By Item (a)  of Lemma \ref{lmaUConQ},  for any 
$\varepsilon >0$,  
\[
 \lim_{n \rightarrow \infty}   \sup_{u \in [\varepsilon, 1]} u^{1/p} { Q}_{Z_n} (u) =  \frac{| \BBE ( X^3)|}{6}   \sup_{u \in [\varepsilon, 1]} u^{1/p} { Q}_{|G^2 -1|} (u) .
\]
Moreover, by Item (b)  of Lemma \ref{lmaUConQ}, 
\[
\lim_{ \varepsilon \rightarrow 0}\sup_{n \geq 1}  \sup_{u \in [0, \varepsilon]} u^{1/p} { Q}_{Z_n} (u)  =  0 \ \text{ and } \ \lim_{ \varepsilon \rightarrow 0} 
\sup_{u \in [0, \varepsilon]} u^{1/p} {  Q}_{|G^2 -1|} (u) =0 .
\]
This ends the proof of the first part  of Item (a) of Corollary \ref{asymptoticconstantsforLambda}. 

We turn now to the proof of the second part  of Item (a) of Corollary \ref{asymptoticconstantsforLambda}. 
With this aim, we first  prove that Item (a)  of Lemma \ref{lmaUConQ} also holds if one replaces  ${ Q}_{Z_n}$ by ${\tilde Q}_{Z_n}$ and ${ Q}_{|G^2 -1|} $  by ${\tilde Q}_{|G^2 -1|} $. 
First, Item (a)  of Lemma \ref{lmaUConQ} implies the pointwise convergence of  ${ Q}_{Z_n}$ to $Q_{|G^2 -1|}  \times \frac{| \BBE ( X^3)|}{6} $ over $]0, 1]$. Next,  by  
Items (a) and (b)  of Lemma \ref{lmaUConQ}, there exists a positive constant $C$, such that for any $u \in ]0,1] $,  $\sup_{n \geq 1}  { Q}_{Z_n} (u) \leq C u^{-1/p}$.  Hence, by  the dominated convergence theorem, for any $u \in [0,1]$, 
\[
 \lim_{n \rightarrow \infty}   u {  \tilde Q}_{Z_n} (u) =  \lim_{n \rightarrow \infty} \int_0^u  { Q}_{Z_n} (v) dv  = u {\tilde Q}_{|G^2 -1|} (u)   \times \frac{| \BBE ( X^3)|}{6} \, .
\]
Since, for any $n$, $u \rightarrow u  {  \tilde Q}_{Z_n} (u) $ is nondecreasing, by the second Dini's theorem, the above convergence is uniform over $[0,1]$. It follows that, for any 
$\varepsilon >0$,  ${\tilde Q}_{Z_n}$ converges uniformly to  ${ \tilde Q}_{|G^2 -1|}    \times \frac{| \BBE ( X^3)|}{6}  $  over $[\varepsilon, 1]$, which implies that 
\[
 \lim_{n \rightarrow \infty}   \sup_{u \in [\varepsilon, 1]} u^{1/p} {\tilde Q}_{Z_n} (u) =  \frac{| \BBE ( X^3)|}{6}   \sup_{u \in [\varepsilon, 1]} u^{1/p} { \tilde Q}_{|G^2 -1|} (u) .
\]
Finally
\[
\lim_{ \varepsilon \rightarrow 0}\sup_{n \geq 1}  \sup_{u \in [0, \varepsilon]} u^{1/p} {\tilde Q}_{Z_n} (u)  =  0 \ \text{ and } \  \lim_{ \varepsilon \rightarrow 0} 
\sup_{u \in [0, \varepsilon]} u^{1/p} { \tilde Q}_{|G^2 -1|} (u) =0 .
\]
This ends the proof of the second part of Item (a) of Corollary \ref{asymptoticconstantsforLambda}.  

\medskip

The proof of Item (b) of Corollary \ref{asymptoticconstantsforLambda} follows the same path as that for Item (a).
Using the fact that ${\bar Z}_n \rightarrow^{\mathcal L} Z $, as  $n \rightarrow \infty$ (see for instance Lemma 2.1 (b) in Rio \cite{Rio11}),
one obtains that  Item (a) of Lemma \ref{lmaUConQ} also holds for the sequence $(Q_{\bar Z_n})_{n\geq 1}$.
Next, from Inequality \eqref{diffZbarZ}, $\tilde Q_{\bar Z_n} \leq \tilde Q_{Z_n} + (h/2)$, which ensures that 
Item (b) of Lemma \ref{lmaUConQ} also holds for the sequence $(\tilde Q_{\bar Z_n})_{n\geq 1}$. Now the end of the  proof of Item (b) of Corollary \ref{asymptoticconstantsforLambda} is exactly the same as for Item (a). \qed

\section{Appendix} \label{sectionAppendix}

\setcounter{equation}{0}

\subsection{Proof of Proposition \ref{PropPoisson}}

By homogeneity, it suffices to prove the result for $\alpha=1$. For any $m >0$,  note that 
\begin{equation} 
\Pi (m)  - m = \sum_{\ell \geq 0} \big ( 2 \Pi ( 2^{\ell} m) -  \Pi ( 2^{\ell+1} m) \big ) 2^{-\ell-1}  \mbox{ in } {\mathbb L}^2 . 
\end{equation}
Indeed, for any positive integer $N$,
\[
\sum_{\ell = 0}^N \big ( 2 \Pi ( 2^{\ell} m) -  \Pi ( 2^{\ell+1} m) \big ) 2^{-\ell-1} =  \Pi (m)  - 2^{-N-1} \Pi ( 2^{N+1} m)
\]
and $\Vert 2^{-N-1} \Pi ( 2^{N+1} m)  - m \Vert_2^2 = 2^{-N-1} m$,  
which converges to zero as $N \rightarrow \infty$. 

Let $ {\tilde U}_{\ell +1,m} :=  2 \Pi ( 2^{\ell} m) -  \Pi ( 2^{\ell+1} m)$ and  $U_{\ell +1,m} :=    \Pi ( 2^{\ell+1} m)$.  Let ${\Phi}$ be the distribution
function of a standard real-valued Gaussian random variable and, for any $\ell \in {\mathbb N}$, let ${\tilde F}_{\ell +1,m}$ be the distribution
function of the conditional law of $ {\tilde U}_{\ell +1,m}$ given  ${U}_{\ell +1,m}$. Let $(\delta_\ell)_{\ell \geq 0}$ be a sequence of i.i.d. random
variables with uniform distribution on $[0,1]$, independent of the Poisson process $(\Pi (t))_{t \geq 0}$. For any $\ell \geq 0$, let 
\[
\xi_{\ell,m}  =  \Phi^{-1}  \Big (  {\tilde F}_{\ell +1,m} (  {\tilde U}_{\ell +1,m} - 0) + \delta_{\ell}  \big ( {\tilde F}_{\ell +1,m} (  {\tilde U}_{\ell +1,m} ) - {\tilde F}_{\ell +1,m} (  {\tilde U}_{\ell +1,m} - 0) \big ) \Big )  . 
\]
By the properties of the conditional quantile transform,  $\xi_{\ell,m}$ is independent of ${U}_{\ell +1,m}$, $\sigma (\delta_{\ell} ,    {\tilde U}_{\ell +1,m}, U_{\ell +1,m})$-measurable and ${\mathcal N} (0,1)$ distributed (for more details see Lemma F1 in \cite{Rio2017}).   Setting ${\mathcal G}_{\ell +1} = \sigma  ( (U_{\ell +k,m}, \delta_{\ell +k}), k \geq 1)$, we note that $  \xi_{\ell,m} $ is independent of ${\mathcal G}_{\ell +1} $. By induction, it follows that the random variables $ ( \xi_{\ell,m} )_{\ell \geq 0}$  are independent and ${\mathcal N}(0,1)$ distributed.


Next, let 
\[
 W(m) :=   \sum_{\ell \geq 0}  \xi_{\ell,m} \sqrt{2^{\ell +1} m }  2^{-\ell-1}  . 
\]
From the above facts, $W(m)$ is $ {\mathcal N} (0,m)$ distributed. 
\par
Next, for any $\ell \geq 0$, let $V_{\ell,m} =  \xi_{\ell,m} \sqrt{2^{l+1}m}   - { \tilde U}_{\ell +1,m}$.  Notice that ${\mathcal L} ( { \tilde U}_{\ell +1,m} |  {\mathcal G}_{\ell +1})$ is the law of $2B -  U_{\ell +1,m}$ with $B \sim {\mathcal B} ( U_{\ell +1,m}, 1/2) $. 
Hence $\BBE (  { \tilde U}_{\ell +1,m} |  {\mathcal G}_{\ell +1} ) =0$ implying that  the random variables $(V_{\ell,m})_{\ell \geq 0}$ are orthogonal in ${\mathbb L}^2$.  Therefore it follows that 
\[
\Vert \Pi(m) - m - W(m)  \Vert^2_2 = \Big \Vert   \sum_{\ell \geq 0}  V_{\ell,m}  2^{-\ell-1} \Big \Vert_2^2 =   \sum_{\ell \geq 0}    2^{-2 ( \ell+1) }  \Vert V_{\ell,m}   \Vert_2^2 ,
\]
from which
\beq \label{lienWassersteinPoissonV}
\Vert \Pi(m) - m - W(m)  \Vert^2_2 \leq  \frac 13 \, \sup_{\ell \geq 0}  \Vert V_{\ell,m}   \Vert_2^2 .
\eeq
We now take care  of the quantities $\Vert V_{\ell,m}   \Vert_2^2$. Set ${\tilde W}_{\ell+1,m} = \xi_{\ell,m} \sqrt{2^{l+1}m}$. Then 
\[
\Vert V_{\ell,m}   \Vert_2^2 = \Vert {\tilde U}_{\ell + 1, m} - {\tilde W}_{\ell+1,m} \Vert_2^2 .
\]

Let  $\mu_0 = 4/3$.  We first consider the case  $\mu \leq \mu_0$. 
In order to shorten the notations, we set 
$\mu = 2^{\ell+1}m$, $U = U_{\ell + 1,m}$, $\tilde U = {\tilde U}_{\ell + 1, m}$ and $\tilde W = {\tilde W}_{\ell+1,m}$.  We first write
\beq \label{equalitycov}
 \BBE (\tilde U \tilde W) = \sum_{n >0} \BBP ( U=n) \BBE ( \tilde U \tilde W \, | \,U=n) \, , 
\eeq
since $\tilde U = 0$ if $U=0$.  Since $\xi_{\ell,m} $ is the conditional quantile transform of $ \tilde U$ conditionally to $U$, according to Exercise 3)-a) page 30 in  Rio \cite{Rio2017}, for any $n>0$,  
\[
  \BBE ( \tilde U  \xi_{\ell,m}   \, | \,U=n)  \geq     \BBE ( \tilde U  \, | \,U=n)  \BBE (  \xi_{\ell,m}   \, | \,U=n)  = 0 \, .
 \]
This implies that  $\BBE ( \tilde U \tilde W ) \geq 0$. Hence
\beq \label{Bornemupetit}
\BBE \bigl( \bigl( \tilde U - \tilde W \bigr)^2 \bigr) \leq \BBE ( \tilde U^2 + \tilde W^2)  = 2 \mu \leq 2 \mu_0 = 8/3.
\eeq

We turn now to the case $\mu > \mu_0$.  In order to bound up the ${\mathbb L}^2$ distance between $\tilde U$ and $\tilde W$, we will give a lower bound on 
 $\BBE (\tilde U \tilde W)$. Indeed 
 \beq \label{L2distanceUW}
\Vert V_{\ell,m}   \Vert_2^2 = \BBE \big (  (\tilde U - \tilde W)^2 \big )= 2\mu  - 
2 \BBE (\tilde U \tilde W),
\eeq
since $\BBE \bigl({\tilde W}^2 \bigr) =  \mu $,  $\BBE (  { \tilde U}^2 \mid  {\mathcal G}_{\ell +1} ) = U_{\ell +1, m}$ and $\BBE ( U_{\ell+1,m} ) = 2^{\ell+1} m$.
\par\smallskip
We rewrite \eqref{equalitycov} as follows: 
\beq \label{1}
 \BBE (\tilde U \tilde W) = \sum_{n >0} \BBP ( U=n)  \BBE ( n^{-1/2} \tilde U \mu^{-1/2}  \tilde W \mid U=n) \sqrt{n\mu}  \, .
 \eeq
 Now 
\beq \label{2}
  \BBE ( n^{-1/2} \tilde U \mu^{-1/2}  \tilde W \, | \,U=n) = 1 - {\textstyle \frac 12} \BBE \bigl(  \bigl( n^{-1/2} \tilde U - \mu^{-1/2} \tilde W \bigr)^2 \mid U=n  \bigr) .
 \eeq
  Note now that  ${\mathcal L} ( { \tilde U} | U=n)$ is the law of $2B -  n$ with $B \sim {\mathcal B} ( n, 1/2) $.  Hence ${\mathcal L} ( n^{-1/2} { \tilde U} | U=n)$ is the law of 
  $n^{-1/2} ( \varepsilon_1 + \cdots + \varepsilon_n)$ where $(\varepsilon_i)_{1 \leq i \leq n}$ are iid r.v.'s with law ${\mathbb P} (\varepsilon_1 =1)= {\mathbb P} (\varepsilon_1 =-1)=1/2$. In addition, the conditional law of $\mu^{-1/2} \tilde W$ given $U=n$ is a ${\mathcal N} (0,1)$. Therefore, by Lemma \ref{TMG} below, 
  \beq \label{3}
  \BBE \Bigl(  \bigl( n^{-1/2} \tilde U - \mu^{-1/2} \tilde W \bigr)^2 \mid  U=n\Bigr) \leq \min \Bigl( \frac{33}{16n} , 2 \bigl( 1 - \sqrt{2/\pi} \, \bigr) \Bigr) .
  \eeq
Starting from \eqref{1} and considering  \eqref{2} and \eqref{3}, it follows that 
  \[
   \BBE (\tilde U \tilde W) \geq  \sum_{n >0} \BBP ( U=n) \sqrt{n\mu} \Bigl( 1 - \min \Bigl(  \frac{33}{32n} , 1- \sqrt{2/\pi} \,\Bigr) \Bigr).
   \]
  The last inequality is equivalent to 
   \[
  \mu^{-1/2}  \BBE (\tilde U \tilde W) \geq \BBE \bigl( \sqrt{U} \bigr) - \BBE \Bigl( \min\Bigl( \frac{33}{32\sqrt{U} } , (1- \sqrt{2/\pi}) \sqrt{U} \Bigr) \Bigr) . 
   \]
  We now claim that 
  \[
   \min\Bigl( \frac{33}{32\sqrt{U} } , (1- \sqrt{2/\pi}) \sqrt{U} \Bigr) \leq c_0 (U+1)^{-1/2}  \text{ with $c_0 = 1.1139$}. 
   \]
Indeed, for $U\leq 5$, $(1- \sqrt{2/\pi}) \sqrt{U} \leq 1.108  (U+1)^{-1/2} $, and, for $U\geq 6$, 
\[
\frac{33}{32\sqrt{U} } \leq \frac{33}{32\sqrt{U+1} } \sqrt{7/6} \leq c_0 (U+1)^{-1/2} .
\]
Hence
 \[
  \mu^{-1/2}  \BBE (\tilde U \tilde W) \geq \BBE \bigl( \sqrt{U} \bigr) - c_0 \, \BBE \Bigl( (U+1)^{-1/2} \Bigr) . 
  \]
Now, recall that $U$ has the Poisson distribution with parameter $\mu$. Therefrom 
\[
\BBE \bigl( (U+1)^{-1/2} \bigr)  =  e^{-\mu} \sum_{n\geq 0} \frac{1}{n!} (n+1)^{-1/2} \mu^n = 
e^{-\mu}  \sum_{n\geq 0} \frac{(n+1)^{1/2}   }{(n+1)!} \mu^n  = \mu^{-1} \BBE \bigl( \sqrt{U} \bigr) . 
\]
The two last inequalities ensure that
\beq \label{Minor1CovtildeUtildeW}
\BBE (\tilde U \tilde W) \geq (\mu - c_0 ) \BBE \bigl( \sqrt{ \mu^{-1} U } \bigr) .
\eeq
\par
We now provide a lower bound on 
 $\BBE \bigl( \sqrt{ \mu^{-1} U } \bigr)$. 
 Let $Z = \mu^{-1} (U-\mu)$. Then $\sqrt{ \mu^{-1} U } = \sqrt{1+Z}$. Now, by the Taylor integral formula at order $4$, for any real $z\geq -1$, 
 \[
 \sqrt{1+z} = 1 + (z/2) - (z^2/8) + (z^3/16) - (5/32) z^4 \int_0^1 (1-t)^3 (1+tz)^{-7/2} dt .
 \]
 Next, for $z\geq -1$, $(1+tz)^{-7/2}  \leq (1-t)^{-7/2}$, from which 
 \[
 -(5/32) z^4 \int_0^1 (1-t)^3 (1+tz)^{-7/2} dt \geq -(5/32) z^4 \int_0^1 (1-t)^{-1/2} dt = -(5/16) z^4.
 \]
 It follows that 
 \[
 \sqrt{\mu^{-1} U} \geq 1 + (Z/2) - (Z^2/8) + (Z^3/16)- (5/16) Z^4
 \]
 Recall now that $U$ has the Poisson distribution with parameter $\mu$. Therefrom $\BBE(Z)=0$, $\BBE (Z^2) = \mu^{-1}$, $\BBE ( Z^3) = \mu^{-2}$
 and $\BBE (Z^4) = \mu^{-3} + 3 \mu^{-2}$. Together with the above inequality, it implies that 
 \beq
 \BBE \bigl( \sqrt{\mu^{-1} U} \bigr) \geq 1 - \frac{1}{8\mu} -\frac{7}{8\mu^2} - \frac{5}{16\mu^3} .
 \eeq
 Combining this lower bound with \eqref{Minor1CovtildeUtildeW} we finally obtain that 
 \[
 \BBE (\tilde U \tilde W) \geq \mu - 1.2389  - 0.7358 \mu^{-1} + 0.6621 \mu^{-2} + 0.3480\mu^{-3} \geq \mu - 1.4055
\]
for any $\mu \geq \mu_0$. The above inequality together with \eqref{L2distanceUW} imply that, for any $\mu\geq \mu_0$,
\beq \label{Bornemugrand}
\BBE \bigl( \bigl( \tilde U - \tilde W \bigr)^2 \bigr) \leq 2 \times 1.4055 = 2.8110. 
\eeq
From \eqref{Bornemupetit}, the above inequality also holds true for $\mu\leq \mu_0$. Finally, combining this upper bound with 
\eqref{lienWassersteinPoissonV}, we  get that 
\beq 
\label{BorneWassersteinPoisson}
\Vert \Pi(m) - m - W(m)  \Vert^2_2 \leq 0.9370.
\eeq
Since $W_2^2 ( \mu_{m,1}, \nu_m) \leq \Vert \Pi(m) - m - W(m)  \Vert^2_2$, the upper bound \eqref{BorneWassersteinPoisson} implies Proposition \ref{PropPoisson}. 
To complete the proof of the proposition, we state and prove Lemma \ref{TMG}.
\begin{Lemma} \label{TMG}
Let $(\varepsilon_i)_{1 \leq i \leq n}$ be iid r.v.'s with law ${\mathbb P} (\varepsilon_1 =1)= {\mathbb P} (\varepsilon_1 =-1)=1/2$,  and $\nu_n$
be the $ {\mathcal N} ( 0,n)$ probability measure.  Then
\[
W_2^2 \big ( P_{\varepsilon_1+ \cdots+\varepsilon_n} , \nu_n \big )  \leq   \min \bigl( 33/16 , 2n \bigl( 1 - \sqrt{2/\pi} \, \bigr) \bigr) .  
\]
\end{Lemma}
\noindent {\it Proof.} Let $G_n \sim {\mathcal N} ( 0,n)$ (hence $P_{G_n} = \nu_n$) and $F_n $ be the c.d.f. of $\varepsilon_1+ \cdots+\varepsilon_n$, and $\Phi$ be the c.d.f. of a $\mathcal N (0,1)$. Define
\[
S_n = F_n^{-1} \bigl (  \Phi  \bigl ( n^{-1/2}  G_n \bigr )  \bigr )  \, .
\]
Then $P_{S_n} = P_{\varepsilon_1+ \cdots+\varepsilon_n}$ and 
$ \BBE \big (   (S_n - G_n )^2 \big ) =W_2^2 ( P_{S_n} , P_{G_n} )$. 
Let $Y = n^{-1/2} G_n$. According to  Theorem 1.1 in \cite{Ma2002}, 
\[
\big |  S_n   -  G_n \big |  \leq  (3/2) + Y^2 /4  .
\]
Now, using the elementary inequality $x^2 \leq \frac{3}{2} |x| + |x| \max  \big ( |x| - \frac{3}{2} , 0 \big )$, we derive that 
\[
\BBE \big (   (S_n - G_n )^2 \big )  \leq \frac{3}{2}   \BBE \big (   \big |    S_n   -  G_n \big |      \big )  +    \BBE  \Big (  \Big [  \frac{3}{2} + \frac{Y^2}{4}\Big ] \frac{Y^2}{4} \Big )   
 \leq  \frac{3}{2}   \BBE \big (   \big | S_n - G_n \big |   \big )  +  \frac{9}{16} .
\]
Next $\BBE \big (   \big | S_n - G_n \big |   \big )  = W_1 ( P_{S_n }, P_{ G_n} )$. Since $\BBE ( \varepsilon_1^2 ) = \BBE (  |\varepsilon_1|^3) = 1$,
according to Theorem  1.1 in \cite{Gold2010}, it follows that
$\BBE \big (   \big | S_n - G_n \big |   \big ) \leq 1$. 
So, overall, 
\[
\BBE \big (   (S_n - G_n )^2 \big ) \leq   (3/2)  +  (9/16)=  33/16 .
\]
It remains to prove that we also have
\beq \label{BoundW2-2}
W_2^2 \big ( P_{\varepsilon_1+ \cdots+\varepsilon_n} , \nu_n \big )  \leq   2n \bigl( 1 - \sqrt{2/\pi} \, \bigr)  .  
\eeq
Let $Y_1, \ldots, Y_n$ be iid r.v.'s with law ${\mathcal N} (0,1)$. Define  $\eta_k$ by 
$\eta_k =  {\bf 1}_{Y_k \geq 0} -   {\bf 1}_{Y_k < 0}$ for $k$  in $[1,n]$.
Then $\eta_1, \ldots, \eta_n$ are iid r.v.'s with law  ${\mathbb P} (\eta_1 =1)= {\mathbb P} (\eta_1 =-1)=1/2$.  Therefrom
\[
W_2^2 \big ( P_{\varepsilon_1+ \cdots+\varepsilon_n} , \nu_n \big )  \leq \BBE  \Big ( \Big ( \sum_{k=1}^n ( \eta_k -  Y_k)  \Big )^2  \Big ) .
\]
Since the r.v.'s $(  ( \eta_k -  Y_k)_{1 \leq k \leq n}$ are iid, 
\[
 \BBE  \Big ( \Big ( \sum_{k=1}^n ( \eta_k -  Y_k)  \Big )^2  \Big )  = n  \BBE (( \eta_1 -  Y_1)^2) .
\]
Next
\beq \label{ERad}
 \BBE (( \eta_1 -  Y_1)^2) =  \BBE (  ( |Y_1| -1)^2  ) = 2 - 2  \BBE ( |Y_1|) = 2 \bigl (  1 - \sqrt{ 2 /\pi }\bigr)  .
\eeq
This ends the proof of \eqref{BoundW2-2} and completes the proof of Lemma \ref{TMG}.  \qed

\par\medskip\noindent

\subsection{Proof of Proposition \ref{linkkappawasserstein}}

The proof  follows immediately from the  general proposition below applied to $\varphi = \psi_x$ where $\psi_x$ is defined in \eqref{defpsi}.

\begin{Proposition} \label{Propdualityconvex}
Let $\varphi $ be  a ${\mathcal C}^1$, non-negative, even and convex function defined on ${\mathbb R}$ such that $\varphi(0)=0$. Assume in addition that  $\varphi'$ is a concave function. Let $\mu$ and $\nu$ two probability laws on the real line with mean zero and respective cumulative distribution functions $F$ and $G$.  Let ${\mathcal F}_\varphi$ be the class of continuously differentiable functions $f : {\mathbb R} \mapsto {\mathbb R}$ such that 
$f(0)=0$, $\Vert f' \Vert_\infty \leq \Vert \varphi' \Vert_\infty$ and 
\[
|f'(x) - f'(y) | \leq 2  \varphi'( |x-y| ) \text{ for any } (x,y) \in  {\mathbb R}^2 .
\] 
Then
\[
\kappa_{ \varphi} ( \mu , \nu)  \leq \zeta_{\varphi} ( \mu, \nu ) := \sup \bigl \{  \mu(f)  -   \nu (f)  \, : \, f \in  {\mathcal F}_\varphi \bigr \} .
\]
\end{Proposition}
\begin{Remark} If $\varphi (x) = |x|^r$ with $r \in ]1, 2] $, we derive that 
$\kappa_{ r} ( \mu , \nu)  \leq 2 r \zeta_r ( \mu , \nu)$,
where $\zeta_r$ is the ideal distance of order $r$ (see for instance \cite[eq. (9.1.2)]{RR98} for its definition).  In this particular case we recover \cite[Theorem 3.1]{Rio09}. 
\end{Remark}

\noindent
{\bf Proof of Proposition \ref{Propdualityconvex}.} 
Using the arguments developed in the beginning of the proof of \cite[Theorem 3.1]{Rio09}, we start by noting that it is enough to prove the proposition for probability laws with strictly positive and smooth densities. Indeed, 
\[
\kappa_{\varphi} ( \mu, \nu )  =   \lim_{\sigma \rightarrow 0}   \kappa_{\varphi} ( \mu_{\sigma}, \nu_{\sigma} ) \text{ and } \zeta_{\varphi} ( \mu, \nu )  =   \lim_{\sigma \rightarrow 0}   \zeta_{\varphi} ( \mu_{\sigma}, \nu_{\sigma} ) \, .
\]
where $\mu_{\sigma} = \mu \ast \phi_{\sigma}. \lambda$ with $\phi_{\sigma}$  the density of the normal law ${\mathcal N} (0, \sigma^2)$. 

So, from now, $\mu$ and $ \nu$  denote two probability laws with distributions function $F$ and $G$, respectively, and we assume that $F$ and $G$ are continuous and strictly increasing from ${\mathbb R}$ to $]0,1[$. For $U$ a r.v. with the uniform distribution over $[0,1]$, the random vector $(F^{-1} (U), G^{-1} (U))$ has respective marginal distributions $\mu$ and $\nu$. Hence, 
to prove the proposition, it is enough to  prove that there exists some function $h$ in ${\mathcal F}_\varphi$ such that, for any $u$ in $]0,1[$,
\begin{equation} \label{p1link}
h( F^{-1}(u)) - h( G^{-1}(u)) \geq  \varphi \big (  F^{-1}(u) -G^{-1}(u) \big )  \, .
\end{equation}
With this aim, let 
\[
H = G-F \text{ and } A = H^{-1} (\{0\}) \, .
\]
Since $H$ is continuous, $A$ is a closed set.  Let $h$ be a function from ${\mathbb R}$ to ${\mathbb R}$  that is derivable and such that 
\[
h'(t) = {\rm sign} H(t)  \varphi' (2 d(t,A)) \, .
\] 
Notice first that $A$ cannot be equal to $\emptyset$.  Indeed, if $A=\emptyset$, then either $F >G$ or $G >F$, which entails that $\mu$ and $\nu$ cannot have the same mean. 

Next, if $[s,t] \cap A =  \emptyset$, then 
\[
|h'(s) - h'(t) | = |  \varphi' (2 d(s,A)) -  \varphi' (2 d(t,A)) | \leq \varphi' ( 2 |t-s| ) \leq 2  \varphi' (  |t-s| ) \, , 
\]
where the first and the second inequalities come from the fact that $\varphi'$ is a concave function, that $\varphi'$ is nondecreasing and $\varphi'(0) =0$. 

Now assume that  $[s,t] \cap A  \neq   \emptyset$. Let $\alpha$ and $\beta$ such that $d(s,A) = |\alpha - s |$ and $d(t,A) = |t - \beta |$. Clearly 
\[
|h'(s) - h'(t) | =   \varphi' (2  |t - \beta |) +  \varphi' (2  |\alpha - s |) |  \, .
\]
Since $\varphi'$ is a concave and nondecreasing function, for any $x$ and $y$ such that $x+y \leq  t-s$, the following inequalities hold: 
\[
 \varphi' ( x) +   \varphi' ( y)  \leq  2 \varphi' \bigl ( (x+y)/2 \bigr) \leq  2 \varphi' \bigl( (t-s)/2\bigr) \, .
\]
This implies that 
\[
|h'(s) - h'(t) | \leq 2  \varphi' (  |t-s| ) \, .
\]
All these considerations show that   $h$ belongs to ${\mathcal F}_\varphi$. The rest of the proof  consists of showing that $h$ satisfies the inequality \eqref{p1link}. 

With this aim, we start by noticing that  $F^{-1}(u) \in A$ is equivalent to $G^{-1} (u) = F^{-1} (u)$.  Indeed,  $F^{-1}(u) \in A$ is equivalent to $G(F^{-1}(u)) = F(F^{-1}(u)) = u$ which in turn is equivalent to   $G^{-1} (u) = F^{-1} (u)$. Consequently, if $F^{-1}(u) \in A$, since $\varphi(0) =0$,
\[
h ( F^{-1}(u)) - h( G^{-1}(u))  =  \varphi \big (  F^{-1}(u) -G^{-1}(u) \big ) = 0.
\]  

Assume from now that $F^{-1}(u) \notin A$. There are different cases.

\noindent  \underline{\textit{Case 1.}}  $F^{-1}(u)  \notin ] \inf A , \sup A[$.  For instance assume that $F^{-1} (u) > \sup A$ and that $ \sup A < + \infty$ (the case where $F^{-1} (u) < \inf A$ can be handled similarly). Let $z =  \sup A$. Since $A$ is closed, $z \in A$. Now $F^{-1}(u) > z$
if and only if  $u > F(z)$. Hence, since $z \in A$, 
\[
G^{-1}(u) > G^{-1}(F(z)) = G^{-1}(G(z)) = z  \, .
\]
Note that for any  $z \in A$,
\[
h ( F^{-1}(u)) - h( G^{-1}(u)) = \int_z^{F^{-1} (u)} h'(t) dt -   \int_z^{G^{-1} (u)} h'(t) dt \, .
\]
Therefore, if $F^{-1}(u)  > G^{-1}(u) $,
\[
h ( F^{-1}(u)) - h( G^{-1}(u)) = \int_{G^{-1}(u)}^{F^{-1} (u)} h'(t) dt  \, ,
\]
and we note that  $G^{-1}(u)  < t <  F^{-1}(u)$ if and only if $F(t) < u < G(t)$ implying that ${\rm sign} (H(t) ) =1$.   Next, if $F^{-1}(u)  < G^{-1}(u) $,
\[
h ( F^{-1}(u)) - h( G^{-1}(u)) = - \int_{F^{-1}(u)}^{G^{-1} (u)} h'(t) dt  \, ,
\]
and  ${\rm sign} (H(t) ) =-1$.  By definition of $h$, if $t > z $, $ d(t,A) =  t-z$ implying that $h'(t) = {\rm sign} H(t)  \varphi' (2 (t-z))$. So, overall,  
\[
D (u):= h( F^{-1}(u)) - h( G^{-1}(u))  =   \int_{\min ( F^{-1}(u), G^{-1} (u))}^{ \max( F^{-1}(u), G^{-1} (u) )}\varphi' ( 2( t-z ))dt \, .
\]

Assume now that $F^{-1} (u) > G^{-1} (u)$ (the case $F^{-1} (u) < G^{-1} (u)$ can be handled similarly).  Then, since $\varphi'$ is nondecreasing and $\varphi'(0) = 0$, 
\[
D (u) =  \int_{G^{-1} (u)-z}^{  F^{-1}(u)-z}\varphi' ( 2 s )ds    \geq  \int_{0}^{  F^{-1}(u)-G^{-1} (u)}\varphi' ( 2 s )ds \geq  
\varphi   ( F^{-1}(u)-G^{-1} (u))   \, .
\]

\medskip

We consider now  the complementary case of  case 1. 

\noindent  \underline{\textit{Case 2.}}  $F^{-1}(u)  \in ] \inf A ,  \sup A[$. Let 
\[
a = \sup( A \cap ]- \infty , F^{-1}(u) ])\  \text{ and  } \  b = \inf( A \cap [F^{-1}(u) , + \infty [)  \, .
\]
Therefore $a < F^{-1}(u) < b $ which implies that $a < G^{-1}(u) < b $. Indeed $a < F^{-1}(u) $ entails that $ u > F(a) $. Hence, since $a \in A$,
\[
G^{-1} (u) >  G^{-1} ( F(a)  ) =   G^{-1} ( G(a)  ) = a   \, .
\]
We proceed similarly to prove that $G^{-1}(u) < b $.   Let then 
\[
I = ]F(a) , F(b) [  =   ]G(a) , G(b) [  \, .
\]
For any  $v \in I$,  $F^{-1} (v) \notin A$, from which $F^{-1} (v)  \neq G^{-1} (v) $. 
Hence the sign of  $F^{-1} (v)  - G^{-1} (v) $ is constant on $I$.

Without loss of generality, we assume in what follows that $F^{-1} (u)  >G^{-1} (u) $ and we set $\delta(u) =  F^{-1} (u)  - G^{-1} (u)$. We have 
\[
D(u) = h( F^{-1}(u)) - h( G^{-1}(u))   = \int_{ G^{-1} (u)}^{ F^{-1} (u)} \varphi' \big ( 2  \big ( ( t-a) \wedge (b-t) \big )  \big )  dt  : = f(G^{-1} (u)) \, , 
\]
where 
\[
f(y) = \int_y^{y+ \delta(u) } \varphi' \big ( 2  \big ( ( t-a) \wedge (b-t) \big )  \big )  dt  \, .
\]
By analyzing the derivative of $f$, one can see that the minimum of $f$ on $[a,b]$ is achieved in $y= a$ or in $y =b- \delta(u)$. Therefore 
\beq \label{UBDu}
D(u)  \geq  \int_{ 0}^{ \delta (u)} \varphi' \big ( 2  \big ( s \wedge (b-a-s) \big )  \big )  ds   \, .
\eeq
Since $\delta (u)  \leq b-a$,
\[
D(u)  \geq \int_{ 0}^{ \delta(u)/2} \varphi'  ( 2  s )  ds +  \int_{ \delta(u)/2}^{ \delta (u)}  \varphi' \big ( 2  \big ( s \wedge (b-a-s) \big )  \big )  ds \, .
\]
We argue now as to get the inequality \eqref{UBDu}. Hence we analyze the derivative of the function $g$ defined by  $g(y) = \int_y^{y+ \delta(u) /2 }   \varphi' \big ( 2  \big ( s \wedge (b-a-s) \big )  \big )  ds$. One can see that the minimum of $g$  is achieved in $y= 0$ or in $y =b- a - \delta(u)/2$. Therefore 
\[
\int_{ \delta(u)/2}^{ \delta (u)}  \varphi' \big ( 2  \big ( s \wedge (b-a-s) \big )  \big )  ds \geq  \int_{ 0}^{ \delta(u)/2} \varphi'  ( 2  s )  ds \, .
\]
Hence
\[
D(u)  \geq   2 \int_{ 0}^{ \delta(u)/2} \varphi'  ( 2  s )  ds    = \varphi ( \delta (u))   \, .
\]
This ends the proof of the proposition. \qed

\subsection{Proof of Lemma \ref{lemmaHtilde}}

Let $C = B-x$. Clearly $\tilde H_{A+B} (x) = \tilde H_{A+C} (0)$. Hence it is enough to prove the lemma in case $x=0$. Note that 
\[
 \tilde H_{A+C} (0) = \inf_{s >0} s^{-1}  \BBE ( (A+C+s)_+ )=  \inf_{v >0 \atop{ w>0}}  (v+w)^{-1} \BBE ( (A+C+v+w)_+ )  .
\]
Similarly
\[
 \tilde H_{A} (t) =   \inf_{v>0}  v^{-1} \BBE ( (A - t +v)_+ ) \ \text{ and  } \  \tilde H_{C} (-t) =   \inf_{w>0}  w^{-1} \BBE ( (C+ t +w)_+ ) .
 \]
Now $  (A+C+v+w)_+ \leq (A - t +v)_+ +  (C+ t +w)_+$. Hence, for any $v,w >0$,
\[
 \tilde H_{A+C} (0)  \leq  \frac{ \BBE ( (A - t +v)_+ )}{v} \frac{v}{v+w}  +   \frac{ \BBE ( (C + t +w)_+ )}{w} \frac{w}{v+w} .
\]
It follows that 
\[
 \tilde H_{A+C} (0)  \leq  \max \bigl ( v^{-1}  \BBE ( (A - t +v)_+ )  ,   w^{-1}  \BBE ( (C + t +w)_+ ) \bigr )  .
\]
Let $\varepsilon >0$ and $v$, $w$ be chosen in such a way that $v^{-1}  \BBE ( (A - t +v)_+ ) \leq \tilde H_{A} (t) + \varepsilon$  and 
$ w^{-1}  \BBE ( (C + t +w)_+ )   \leq \tilde H_{C} (-t) + \varepsilon$. Then 
$\tilde H_{A+C} (0)  \leq \max (  \tilde H_{A} (t)  ,  \tilde H_{C} (-t) ) + \varepsilon$. 
The result follows by letting $\varepsilon $ tend to $0$. \qed

\end{document}